\documentclass[11pt]{amsart}
\usepackage[margin=1in,letterpaper]{geometry}

\usepackage{amssymb,amsfonts,amsmath,bbm,mathrsfs,stmaryrd,wasysym}

\usepackage{xcolor}
\usepackage[colorlinks,
             linkcolor=black!75!red,
             citecolor=blue!50!black,
             pdftitle={A combinatorial universal star-product},
             pdfauthor={Theo Johnson-Freyd},
             pdfsubject={quantization},
             pdfkeywords={quantization,diagrams},
             pdfproducer={pdfLaTeX},
             pdfpagemode=None,
             bookmarksopen=true
             bookmarksnumbered=true]{hyperref}

\usepackage{amsthm}

\newtheorem{proposition}[subsection]{Proposition}
\newtheorem{theorem}[subsection]{Theorem}

\newtheorem{corollary}[subsection]{Corollary}
\newtheorem{lemma}[subsection]{Lemma}

\newtheorem{definition}[subsection]{Definition}

\theoremstyle{definition}

\usepackage{tikz}
\usetikzlibrary{arrows,decorations.pathmorphing,decorations.pathreplacing,decorations.markings,shapes.geometric,through,fit,shapes.symbols}
\tikzset{
    dot/.style={circle,draw,fill,inner sep=1.25pt},
    intdot/.style={circle,draw,fill=gray,inner sep=2pt},
    opendot/.style={circle,draw,inner sep=1pt},
    onearrow/.style={postaction={decorate}, decoration={markings,mark=at position .6 with {\arrow[draw,line width=1pt]{>}}}},
    inversearrow/.style={postaction={decorate}, decoration={markings,mark=at position .45 with {\arrow[draw,line width=1pt]{<}}}},
    twoarrows/.style={draw, postaction={decorate}, decoration={markings,mark=at position .35 with {\arrow[draw,line width=1pt]{>}},mark=at position .75 with {\arrow[draw,line width=1pt]{>}}}},
    twoarrowsempty/.style={postaction={decorate}, decoration={markings,mark=at position .3 with {\arrow[draw,line width=1pt]{>}},mark=at position .7 with {\arrow[draw,line width=1pt]{>}}}},
    inversetwoarrows/.style={draw, postaction={decorate}, decoration={markings,mark=at position .35 with {\arrow[draw,line width=1pt]{<}},mark=at position .7 with {\arrow[draw,line width=1pt]{<}}}},
    squiggly/.style={draw, decorate,decoration={snake,amplitude=.3mm,segment length=2mm}},
    fastsquiggly/.style={draw, decorate,decoration={snake,amplitude=.3mm,segment length=1mm}},
    inversesquiggly/.style={draw, decorate,decoration={snake,amplitude=.2mm,segment length=2mm},postaction={decorate,decoration={markings,mark=at position .45 with {\arrow[draw,line width=1pt]{<}}}}},
    degreeshift/.style={onearrow,dashed},
    octarine/.style={postaction={decorate,decoration={markings,mark=at position .6 with {\arrow[draw,line width=1pt]{>}}}}},
    tensor/.style={draw,double,double distance=1.5pt},
}

\newcommand\mult{\,\tikz[baseline=(basepoint)]{ 
    \path (0,0) coordinate (basepoint) (0,4pt) node[dot] {};
    \draw[](-6pt,-4pt) -- (0,4pt);
    \draw[](6pt,-4pt) -- (0,4pt);
    \draw[](0,4pt) -- (0,12pt);
  }\,}
\newcommand\comult{\,\tikz[baseline=(basepoint)]{ 
    \path (0,0) coordinate (basepoint) (0,4pt) node[dot] {};
    \draw[](0,-4pt) -- (0,4pt);
    \draw[](0,4pt) -- (-6pt,12pt);
    \draw[](0,4pt) -- (6pt,12pt);
  }\,}

\newcommand\define[1]{{\em #1}}
\newcommand\cat[1]{\textmd{\textsc{#1}}}

\newcommand\Ff{{\mathcal F}}

\newcommand\NN{{\mathbb N}}
\newcommand\QQ{{\mathbb Q}}
\newcommand\RR{{\mathbb R}}
\renewcommand\SS{{\mathbb S}}

\newcommand\ZZ{\mathbb Z}

\newcommand\B{{\mathbf B}}
\newcommand\C{{\mathrm C}}

\DeclareMathOperator{\End}{End}

\DeclareMathOperator{\DGVect}{\cat{DGVect}}

\DeclareMathOperator{\Chains}{Chains}
\DeclareMathOperator{\Cochains}{Cochains}

\DeclareMathOperator{\Homology}{H}
\renewcommand\H\Homology

\DeclareMathOperator{\Frob}{Frob}

\DeclareMathOperator\shBD{shBDF}
\newcommand\shBDF\shBD

\DeclareMathOperator{\qloc}{QLoc}
\DeclareMathOperator{\inv}{inv}

\DeclareMathOperator{\di}{di}
\DeclareMathOperator{\pr}{pr}
\DeclareMathOperator{\op}{op}

\newcommand\h{\mathbf{h}}
\newcommand\sh{\mathbf{sh}}

\DeclareMathOperator{\diag}{diag}

\newcommand\mono{\hookrightarrow}
\newcommand\epi{\twoheadrightarrow}
\newcommand\onto\epi

\newcommand\<\langle
\renewcommand\>\rangle

\newcommand\Circ{\ocircle}

\newcommand\shriek{{\textnormal{\textexclamdown}}}

\title[$\Chains(\RR)$\MakeLowercase{ does not admit a geometrically meaningful properadic homotopy Frobenius algebra structure}]{\texorpdfstring{$\mathbf{C\lowercase{hains}}(\RR)$\MakeLowercase{ does not admit a geometrically meaningful properadic homotopy \uppercase{F}robenius algebra structure}}{Chains(R) does not admit a geometrically meaningful properadic homotopy Frobenius algebra structure}}
\author{Theo Johnson-Freyd}

\begin{document}

\begin{abstract}
  The embedding $\Chains_{\bullet}(\RR) \mono \Cochains^{1-\bullet}(\RR)$ as the compactly supported cochains might lead one to expect $\Chains_{\bullet}(\RR)$ to carry a nonunital commutative Frobenius algebra structure, up to a degree shift and some homotopic weakening of the axioms.  We prove that under reasonable ``locality'' conditions, a cofibrant resolution of the dioperad controlling nonunital shifted-Frobenius algebras does act on $\Chains_{\bullet}(\RR)$, and in a homotopically-unique way.  But we prove that this action does not extend to a homotopy Frobenius action at the level of properads or props.  This gives an example of a geometrically meaningful algebraic structure on homology that does not lift in a geometrically meaningful way to the chain level.
\end{abstract}

\maketitle

\section{Introduction}

A basic tenet of algebraic topology is that algebraic structures on the homology of a space should come from  algebraic structures at the chain level,  with the understanding that equations are weakened to homotopy equivalences.  Given some structure on homology, one then can pose the following questions: What is the appropriate weakening from equation to homotopy necessary to define a structure on chains?  What is that structure at the chain level?

Here is an unsatisfying answer.  Suppose that $M$ is some space, and let $\H_{\bullet}(M)$ denote its homology with coefficients in a field.  Suppose that $P$ is a properad (or operad, or dioperad, or \dots) that acts on $\H_{\bullet}(M)$, and let $\h P$ denote any cofibrant replacement of~$P$.  Let $\C_{\bullet}(M)$ denote the complex of chains for your favorite chain model.  There are myriad ways to choose a deformation retraction of $\C_{\bullet}(M)$ onto $\H_{\bullet}(M)$.  Any such choice determines, via homotopy transfer theory, an action of $\h P$ on $\C_{\bullet}(M)$ that induces the action of $P$ on $\H_{\bullet}(M)$.  The space of choices required to carry out this procedure is contractible, which is to say there is no choice at all.

One reason the above answer is unsatisfying is that it doesn't lead to any further insight into the topology of $M$ than what was already available  from the action of $P$ on $\H_{\bullet}(M)$.  Rather, when one asks to lift the $P$-action to the chain level, one usually means that the action should be ``geometrically meaningful,'' preferably with some notion of ``locality'' built in.  For example, the cohomology $\H^{\bullet}(M)$ is naturally a commutative algebra; the multiplication corresponds to the diagonal map $M \mono M \times M$.  One can look for cochain-level multiplications $\C^{\bullet}(M) \otimes \C^{\bullet}(M) \to \C^{\bullet}(M)$ that respect some notion of ``locality,'' and extend such a multiplication to an action by some cofibrant replacement $\h\mathrm{Com}$ of the commutative operad $\mathrm{Com}$.  By using a deformation retraction of $\C^{\bullet}(M)$ onto $\H^{\bullet}(M)$ and some homotopy transfer theory, one can then use the $\h\mathrm{Com}$-structure on $\C^{\bullet}(M)$ to build  an $\h\mathrm{Com}$-algebra structure on $\H^{\bullet}(M)$.  This structure will begin with the commutative multiplication, but include more data, namely the \define{Massey products}.

However, as this paper demonstrates, not all algebraic structures on homology lift in a geometric way to the chain level.  Define a $\Frob_{1}$-algebra to be a graded commutative Frobenius algebra in which the comultiplication has homological degree $0$ but the multiplication has homological degree $-1$; the main example is the homology $\H_{\bullet}(S^{1})$ of a circle.  We will focus on  ``open and coopen'' Frobenius algebras, which need not have unit or counit, and can be infinite-dimensional.  It is reasonable to expect the chains $\C_{\bullet}(\RR)$ on the line to support an interesting $\Frob_{1}$-algebra structure: if we take our chain model to be the complex $\C_{\bullet}(\RR) = \Omega^{1-\bullet}_{\mathrm{cpt},\mathrm{sm}}(\RR)$ of compactly-supported smooth de Rham forms, then we get a non-unital degree-$(-1)$ multiplication which is strictly (anti)commutative and associative; if we take our chain model to be the complex $\C_{\bullet}(\RR) = \Omega^{1-\bullet}_{\mathrm{cpt},\mathrm{dist}}(\RR)$ of compactly supported distributional de Rham forms, then we get a degree-$0$ comultiplication which is strictly cocommutative and coassociative.
Indeed, we will prove that any reasonable model of $\C_{\bullet}(\RR)$ supports a canonical (up to a contractible space of choices) $\h\Frob_{1}$-algebra structure generalizing these (co)multiplications, provided $\h\Frob_{1} = \h^{\di}\Frob_{1}$ is resolved as a \define{dioperad} (meaning only tree-like compositions are used).  But we will prove that $\C_{\bullet}(\RR)$ does not support any geometrically meaningful action of the \define{properadic} (graph-like compositions) cofibrant resolution $\h^{\pr}\Frob_{1}$.  This contradicts the main result of~\cite{Wilson2007}.

\subsection{Outline} In Section~\ref{section.qloc} we define the notion of \define{quasilocality} that we take as a minimal requirement for a chain-level structure to be ``geometrically meaningful.''  
As explained in Lemma~\ref{lemma.cochains}, quasilocal operations extend canonically from chains to ``non-compactly supported chains,'' which is a degree-shifted model of cochains.  Such extensions will allow us to tell apart operations that are homotopically equivalent among all operations on $\C_{\bullet}(\RR)$ --- the wedge product of compactly supported forms from the zero product, for example.

Section~\ref{section.diproperads} recalls the basic theory of dioperads and properads, and introduces our main example, the (di/pr)operad $\Frob_{1}$ controlling nonunital noncounital commutative Frobenius algebras in which the multiplication has homological degree $-1$.  The Koszulity of $\Frob_{1}$ is proven in Section~\ref{section.Koszul}, and used to compute a small cofibrant replacement $\sh\Frob_{1}$.  The main results are in Section~\ref{section.mainthms}: Theorem~\ref{thm.dioperadic} constructs a contractible space of quasilocal translation-invariant actions of the dioperad $\sh^{\di}\Frob_{1}$ on $\Chains_{\bullet}(\RR)$, and Theorem~\ref{thm.properadic} proves that the properad $\sh^{\pr}\Frob_{1}$ does not act quasilocally in a way that induces both the comultiplication on homology and (via the extension to cochains) the multiplication on (degree-shifted) cohomology.

\subsection{Acknowledgements} Conversations with Kevin Costello, Gabriel Drummond-Cole, Joey Hirsh, Nick Rozenblyum, Thomas Willwacher, Scott Wilson, and Bruno Vallette were both enjoyable and valuable.  The anonymous referee's suggestions led to important improvements in the exposition.  This work is supported by the NSF grant DMS-1304054.

\subsection{Conventions}

We work over a ground field of characteristic $0$, which we will call $\QQ$, even though the reader may in fact choose to use de Rham forms with coefficients in $\RR$.  We always use \define{homological} conventions --- the differential has homological degree $-1$ --- and denote the category of chain complexes by $\DGVect$.  We use the usual Koszul sign rules: the canonical isomorphism $V \otimes W \cong W \otimes V$, for $V,W\in \DGVect$, sends $v \otimes w \mapsto (-1)^{(\deg v)(\deg w)}w\otimes v$ if $v$ and $w$ are homogeneous.  We let $[n]$ denote the one-dimensional graded vector space satisfying $\dim [n]_{n} = 1$ and $\dim [n]_{\bullet} = 0$ if $\bullet\neq n$, and we shift chain complexes by $V[n] = V \otimes [n]$; note that this introduces signs to formulas involving homogeneous elements.  Our primary references for the theory of properads, including their homotopy theory and Koszul duality, are~\cite{MR2320654,MR2560406,MR2572248}.

\section{Quasilocality}\label{section.qloc}

We fix a model of chains $\C_{\bullet}(\RR^{n})$, defined (at least) on $n$-dimensional Euclidean space.  Here are some models that work in our construction (many others work as well, and the   reader may always check their favorite chain model):
\begin{enumerate}
  \item $\C_{\bullet}(\RR^{n}) =  \Omega^{n-\bullet}_{\mathrm{cpt},\mathrm{sm}}(\RR^{n})$, the complex of smooth compactly-supported de Rham forms, shifted in homological degree to make it a model of chains.
  \item $\C_{\bullet}(\RR^{n}) =  \Omega^{n-\bullet}_{\mathrm{cpt,dist}}(\RR^{n})$, the complex of distributional compactly-supported de Rham forms, similarly shifted.
  \item $\C_{\bullet}(\RR^{n}) = \operatorname{span}_{\QQ}\{ c_{\vec z}, \vec z\in (\ZZ \oplus \ZZ+\frac12)^{n} \}$ with $\deg c_{\vec z} = $ number of entries in $\vec z$ in $\ZZ+\frac12$, and differential extending $\partial c_{z+\frac12} = c_{z+1}-c_{z}$ for $z\in \ZZ$ via the Leibniz rule $\partial c_{(\vec x,\vec y)} = \partial c_{\vec x}\otimes c_{\vec y} + (-1)^{\deg c_{\vec x}} c_{\vec x}\otimes\partial c_{\vec y}$, where we identify $c_{(\vec x,\vec y)} = c_{\vec x}\otimes c_{\vec y}$.  This is the complex of cellular chains for the cubulation of $\RR^{n}$ given by slicing along the hyperplanes $\RR^{k-1}\times \{z\} \times \RR^{n-k}$ for $z\in \ZZ$ and $k = 1,\dots,n$.
\end{enumerate}
For these models, one has a canonical isomorphism $\C_{\bullet}(\RR^{m}) \otimes \C_{\bullet}(\RR^{n}) \cong \C_{\bullet}(\RR^{m+n})$; the tensor product is the algebraic tensor product if cellular chains are used, but the projective tensor product for the de Rham models.  

For each of these models there are two corresponding cochain models.  Denote by $\C^{-\bullet}(\RR^{n})$ the model of cochains that receives an inclusion $\C_{\bullet}(\RR^{n}) \mono \C^{n-\bullet}(\RR^{n})$ of complexes identifying the chains as the compactly supported cochains (the minus sign in $\C^{-\bullet}$ is to emphasize that we will always use homological conventions, hence  cohomology is supported in negative degrees).  Denote by $(\C_{\bullet}(\RR^{n}))^{\vee}$ the cochain model given as the linear dual to the chosen chain model.  In our examples these are:
\begin{enumerate}
  \item When $\C_{\bullet}(\RR^{n}) =  \Omega^{n-\bullet}_{\mathrm{cpt},\mathrm{sm}}(\RR^{n})$, 
  we have $\C^{-\bullet}(\RR^{n}) = \Omega^{-\bullet}_{\mathrm{sm}}(\RR^{n})$, and $(\C_{\bullet}(\RR^{m}))^{\vee} = \Omega^{-\bullet}_{\mathrm{dist}}(\RR^{n})$.
  \item When $\C_{\bullet}(\RR^{n}) =  \Omega^{n-\bullet}_{\mathrm{cpt,dist}}(\RR^{n})$, 
  we have $\C^{-\bullet}(\RR^{n}) = \Omega^{-\bullet}_{\mathrm{dist}}(\RR^{n})$, and $(\C_{\bullet}(\RR^{m}))^{\vee} = \Omega^{-\bullet}_{\mathrm{sm}}(\RR^{n})$.
  \item When $\C_{\bullet}(\RR^{n}) = $ cellular chains   for the cubulation along the hyperplanes $\RR^{k-1}\times \{z\} \times \RR^{n-k}$ for $z\in \ZZ$, we take $\C^{n-\bullet}(\RR^{n})$ to be the completion of $\C_{\bullet}(\RR^{n})$ in which one may take infinite sums of basis cells provided no cell appears with an infinite coefficient.  The dual complex $(\C_{\bullet}(\RR^{m}))^{\vee}$ is formed in the same way as $\C^{-\bullet}(\RR^{n})$, except using the cubulation given by dividing $\RR^{n}$ along hyperplanes $\RR^{k-1}\times \{z+\frac12\} \times \RR^{n-k}$ for $z\in \ZZ$.
\end{enumerate}

Any linear map $f: \C_{\bullet}(\RR^{m}) \to \C_{\bullet}(\RR^{n})$ (satisfying, in the de Rham models, some continuity properties) defines a cochain on $\RR^{m+n}$, called its \define{graph}:
using as necessary the projective tensor product, one has
 $f\in (\C_{\bullet}(\RR^{m}))^{\vee}\otimes \C_{\bullet}(\RR^{n}) \mono (\C_{\bullet}(\RR^{m}))^{\vee}\otimes \C^{n-\bullet}(\RR^{n})$, which is, up to a shift of homological degrees, a cochain model on $\RR^{m+n}$.  For the de Rham models, $\operatorname{graph}(f)$ is the integral kernel of $f$, and for cellular chains $\operatorname{graph}(f) \in \C^{\bullet}(\RR^{m+n})$ simply records the matrix coefficients of $f$.  In particular, it makes sense to talk about the \define{support} of $\operatorname{graph}(f)$.

Let $\diag: \RR \mono \RR^{m+n}$ denote the diagonal embedding.  For $\ell \in \RR_{>0}$, let $B_{\ell}(\diag(\RR))$ denote the closed tubular neighborhood of $\diag(\RR)$ of radius $\ell$.

\begin{definition}\label{defn.qloc}
  A linear map $f: \C_{\bullet}(\RR^{m}) \to \C_{\bullet}(\RR^{n})$, for $m,n > 0$, is \define{$\ell$-quasilocal} if its graph is supported in $B_{\ell}(\diag(\RR))$, and \define{quasilocal} if it is $\ell$-quasilocal for some $\ell > 0$.  It is clear that the boundary $[\partial,f] = \partial \circ f - (-1)^{\deg f}f\circ \partial$ of $f$ is $\ell$-quasilocal if $f$ is --- its graph is $ \partial \bigl(\operatorname{graph}(f)\bigr)$.  The complex of all quasilocal maps $\C_{\bullet}(\RR^{m}) \to \C_{\bullet}(\RR^{n})$ is denoted $\qloc(m,n)$.
  
  A linear map $f: \C_{\bullet}(\RR^{m}) \to \C_{\bullet}(\RR^{n})$, for $m,n > 0$, is \define{translation-invariant} if it is equivariant for the translation action of $\RR$ (for the de Rham models) or $\ZZ$ (for cellular chains) acting on $\RR$.  The subcomplex of $\qloc(m,n)$ consisting of translation-invariant maps is denoted $\qloc^{\inv}(m,n)$.
\end{definition}

Quasilocality provides a good chain-level version of ``locality'' for the purposes of intersection theory, since often one must perturb things slightly (thus ruining any ``strict'' locality) to make  intersections well-defined.  Translation-invariance is a reasonable request for any geometrically-meaningful structure on $\RR$; we achieve translation invariance in Theorem~\ref{thm.dioperadic}, and do not demand it in the ``no-go'' Theorem~\ref{thm.properadic}.

An important side effect of quasilocality is that quasilocal maps extend from chains to cochains:

\begin{lemma}\label{lemma.cochains}
  Any
  $f \in \qloc(m,n)$ defines a map $\C^{1-\bullet}(\RR)^{\otimes m} \to \C^{1-\bullet}(\RR)^{\otimes n}$, in addition to the map $\C_{\bullet}(\RR)^{\otimes m} \to \C_{\bullet}(\RR)^{\otimes n}$ given in Definition~\ref{defn.qloc}.  These two versions of $f$ are intertwined by the inclusion $\C_{\bullet}(\RR) \mono \C^{1-\bullet}(\RR)$.
  
    In particular, if $f$ is $\partial$-closed, then it defines a map $\H_{\bullet}(f) : \QQ = \H_{\bullet}(\RR)^{\otimes m} \to \H_{\bullet}(\RR)^{\otimes n} = \QQ$ and also a map $\H^{1-\bullet}(f) : \QQ[m] = \H^{1-\bullet}(\RR)^{\otimes m}\to \H^{1-\bullet}(\RR)^{\otimes n} = \QQ[n]$. 
\end{lemma}

\begin{proof}
  A cochain $f$ on $\RR^{m+n}$ (for the cochain model $(\C_{\bullet}(\RR^{m}))^{\vee}\otimes \C^{-\bullet}(\RR^{n})$) with support $\operatorname{supp}(f)$ is the graph of a linear map $\C_{\bullet}(\RR^{m}) \to \C_{\bullet}(\RR^{n})$ if and only if for every bounded set $U \subseteq \RR^{m}$, the intersection $(U \times \RR^{n})\cap \operatorname{supp}(f)$ is bounded.  Similarly, $f$ defines a map $\C^{m-\bullet}(\RR^{m}) \to \C^{n-\bullet}(\RR^{n})$ if and only if $(\RR^{m}\times V) \cap  \operatorname{supp}(f)$ is bounded for all bounded $V \subseteq \RR^{n}$.  When $m,n>0$, both boundedness conditions hold when $\operatorname{supp}(f) \subseteq B_{\ell}(\diag(\RR))$.
\end{proof}

Finally, the space of quasilocal maps has a very manageable homology:

\begin{proposition} \label{homologyofqloc}
  The homology of $\qloc(m,n)$ and $\qloc^{\inv}(m,n)$ are: 
  \begin{gather*}
   \dim\H_{\bullet}\bigl(\qloc(m,n),\partial\bigr) = \begin{cases} 1, & \bullet =  -m+1 \\ 0, & \text{otherwise.} \end{cases} \\
   \dim\H_{\bullet}\bigl(\qloc^{\inv}(m,n),\partial\bigr) = \begin{cases} 1, & \bullet = -m \text{ or } -m+1 \\ 0, & \text{otherwise.} \end{cases} 
   \end{gather*}
\end{proposition}

\begin{proof}
  The neighborhood $B_{\ell}(\diag(\RR))$ contracts onto $\diag(\RR)$ in a translation-invariant way.  Thus the complex of 
  cochains supported in $B_{\ell}(\diag(\RR))$ has, up to a shift of degrees, the cohomology of $\RR$,  the complex of
  translation-invariant cochains supported in $B_{\ell}(\diag(\RR))$ has the cohomology of a circle, and the inclusions as $\ell$ increases are quasiisomorphisms.  
  
  To compute the degree shifts, it is convenient to think in terms of one of the de Rham models.  The non-zero class in $\C^{-\bullet}(\RR)$ is represented by a form supported in  $B_{\ell}(\diag(\RR)) \subseteq \RR^{m+n}$ of form degree $m+n-1$ (e.g.\ one may take the constant function $1$ in the direction parallel to $\diag(\RR)$ tensored with a top form of total volume $1$ in the transverse $\RR^{m+n-1}$).  Thus the non-zero class in $\qloc(m,n)$ has its graph in homological degree $1-(m+n)$ in $(\C_{\bullet}(\RR^{m}))^{\vee}\otimes \C^{-\bullet}(\RR^{n})$, or in homological degree $1-(m+n) + n = 1-m$ in the shifted complex $(\C_{\bullet}(\RR^{m}))^{\vee}\otimes \C^{n-\bullet}(\RR^{n})\supseteq \qloc(m,n)$.  The other cohomology class in $S^{1}$ is in cohomological degree one higher, which is to say in homological degree one lower.
\end{proof}

\section{Dioperads and properads} \label{section.diproperads}

Dioperads were introduced in~\cite{MR1960128} and properads in~\cite{MR2320654}.  Both provide frameworks in which to axiomatize algebraic structures with many-to-many operations.  There are many equivalent definitions; we will use the following.

\begin{definition}
  Let $\SS$ denote the groupoid of finite sets and bijections.  An \define{$\SS$-bimodule} is a functor $P : \SS^{\op}\times \SS \to \DGVect$.  Thus, the data of an $\SS$-bimodule is a collection of chain complexes $P(m,n)$ for $(m,n)\in \NN^{2}$, along with, for each $(m,n)$, an action on $P(m,n)$ of $\SS_{m}^{\op} \times \SS_{n}$, where $\SS_{n}$ denotes the symmetric group on $n$ letters.
\end{definition}

\begin{definition} \label{defn.diproperad}
  A \define{nonunital properad} is an $\SS$-bimodule along with, for every tuple of finite sets $\mathbf{m}_{1},\mathbf{m}_{2},\mathbf{n}_{1},\mathbf{n}_{2},\mathbf{k}$ with $\mathbf{k}$ nonempty, a \define{composition} map:
  $$
  \begin{tikzpicture}[baseline=(basepoint),yscale=1.25]
  \path(0,.25) coordinate (basepoint);
  \draw (-.5,.5) node[circle,draw,inner sep=3,fill=gray] (delta1) {};
  \draw (.5,0) node[circle,draw,inner sep=3,fill=gray] (delta2) {};
  \draw[onearrow]  (-.25,-.5) .. controls +(0,.25) and +(.25,-.5) .. (delta1);
  \draw[onearrow] (-.75,-.5) .. controls +(0,.25) and +(-.25,-.25) .. (delta1);
  \draw[onearrow]  (.25,-.5) .. controls +(0,.25) and +(-.25,-.25) .. (delta2);
  \draw[onearrow] (.75,-.5) .. controls +(0,.25) and +(.25,-.25) .. (delta2);
  \draw[onearrow] (delta1) .. controls +(-.25,.25) and +(0,-.25) .. (-.85,1.1);
  \draw[onearrow] (delta1) .. controls +(.25,.25) and +(0,-.25) .. (-.15,1.1);
  \draw[onearrow] (delta2) .. controls +(-.25,.5) and +(0,-.25) .. (.15,1.1);
  \draw[onearrow] (delta2) .. controls +(.25,.25) and +(0,-.25) .. (.85,1.1);
  \draw[] (-.45,1) node {$\scriptstyle \dots$} (.55,1) node {$\scriptstyle \dots$};
  \draw[] (-.45,-.35) node {$\scriptstyle \dots$} (.55,-.35) node {$\scriptstyle \dots$};
  \draw[] (delta2) ..controls +(-.5,.125) and +(.25,-.25) .. (delta1);
  \draw[] (delta2) ..controls +(-.25,.25) and +(.5,-.125) .. (delta1);
  \draw[] (0,.25) node[anchor=base] {$\scriptscriptstyle \cdots$};
    \draw[decorate,decoration=brace] (-.9,1.15) -- node[auto] {$\scriptstyle \mathbf n_{2}$} (-.1,1.15);
    \draw[decorate,decoration=brace] (.1,1.15) -- node[auto] {$\scriptstyle \mathbf n_{1}$} (.9,1.15);
    \draw[decorate,decoration=brace] (-.2,-.55) -- node[auto] {$\scriptstyle \mathbf m_{2}$} (-.8,-.55);
    \draw[decorate,decoration=brace] (.8,-.55) -- node[auto] {$\scriptstyle \mathbf m_{1}$} (.2,-.55);
    \draw[decorate,decoration=brace] (.25,.35) -- node[anchor=north] {$\scriptstyle \mathbf k$} (-.2,.15);
\end{tikzpicture}
:
   P\bigl(\mathbf m_{1},\mathbf k \sqcup \mathbf n_{1} \bigr) \otimes P\bigl(\mathbf m_{2} \sqcup \mathbf k,\mathbf n_{2}\bigr) \to P\bigl(\mathbf m_{1} \sqcup \mathbf m_{2}, \mathbf n_{1} \sqcup \mathbf n_{2} \bigr) $$
   The composition maps should be compatible with the $\SS$-bimodule structure in the obvious way from the picture.  Moreover, we demand an \define{associativity} condition, which is actually four conditions for the types of connected directed  graphs with three vertices and no directed cycles:
  $$
  
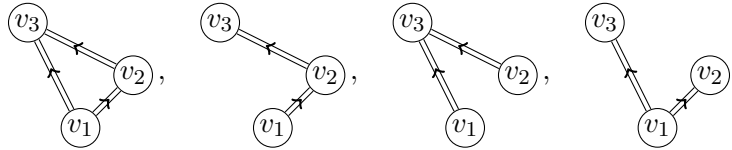
\begin{tikzpicture}[baseline=(v2.base)]
    \path (0,0) node[draw,circle,inner sep = 1pt] (v1) {$v_{1}$};
    \path (.7,.7) node[draw,circle,inner sep = 1pt] (v2) {$v_{2}$};
    \path (-.7,1.4) node[draw,circle,inner sep = 1pt] (v3) {$v_{3}$};
    \draw[tensor,onearrow] (v1) -- (v2);
    \draw[tensor,onearrow] (v1) -- (v3);
    \draw[tensor,onearrow] (v2) -- (v3);
  \end{tikzpicture}\,,\quad
  \begin{tikzpicture}[baseline=(v2.base)]
    \path (0,0) node[draw,circle,inner sep = 1pt] (v1) {$v_{1}$};
    \path (.7,.7) node[draw,circle,inner sep = 1pt] (v2) {$v_{2}$};
    \path (-.7,1.4) node[draw,circle,inner sep = 1pt] (v3) {$v_{3}$};
    \draw[tensor,onearrow] (v1) -- (v2);
    \draw[tensor,onearrow] (v2) -- (v3);
  \end{tikzpicture}\,,\quad
  \begin{tikzpicture}[baseline=(v2.base)]
    \path (0,0) node[draw,circle,inner sep = 1pt] (v1) {$v_{1}$};
    \path (.7,.7) node[draw,circle,inner sep = 1pt] (v2) {$v_{2}$};
    \path (-.7,1.4) node[draw,circle,inner sep = 1pt] (v3) {$v_{3}$};
    \draw[tensor,onearrow] (v1) -- (v3);
    \draw[tensor,onearrow] (v2) -- (v3);
  \end{tikzpicture}\,,\quad
  \begin{tikzpicture}[baseline=(v2.base)]
    \path (0,0) node[draw,circle,inner sep = 1pt] (v1) {$v_{1}$};
    \path (.7,.7) node[draw,circle,inner sep = 1pt] (v2) {$v_{2}$};
    \path (-.7,1.4) node[draw,circle,inner sep = 1pt] (v3) {$v_{3}$};
    \draw[tensor,onearrow] (v1) -- (v2);
    \draw[tensor,onearrow] (v1) -- (v3);
  \end{tikzpicture}
  $$
  The reader is invited to spell out the details of the associativity equations; note that for the last two, one must reverse the order of two factors in a tensor product, and this introduces signs.
  
  A \define{nonunital dioperad} is as above, but the only compositions that are defined are when $\mathbf k$ is a set of size $1$:
  $$
  \begin{tikzpicture}[baseline=(basepoint),yscale=1.25]
  \path(0,.25) coordinate (basepoint);
  \draw (-.5,.5) node[circle,draw,inner sep=3,fill=gray] (delta1) {};
  \draw (.5,0) node[circle,draw,inner sep=3,fill=gray] (delta2) {};
  \draw[onearrow]  (-.25,-.5) .. controls +(0,.25) and +(.25,-.5) .. (delta1);
  \draw[onearrow] (-.75,-.5) .. controls +(0,.25) and +(-.25,-.25) .. (delta1);
  \draw[onearrow]  (.25,-.5) .. controls +(0,.25) and +(-.25,-.25) .. (delta2);
  \draw[onearrow] (.75,-.5) .. controls +(0,.25) and +(.25,-.25) .. (delta2);
  \draw[onearrow] (delta1) .. controls +(-.25,.25) and +(0,-.25) .. (-.85,1.1);
  \draw[onearrow] (delta1) .. controls +(.25,.25) and +(0,-.25) .. (-.15,1.1);
  \draw[onearrow] (delta2) .. controls +(-.25,.5) and +(0,-.25) .. (.15,1.1);
  \draw[onearrow] (delta2) .. controls +(.25,.25) and +(0,-.25) .. (.85,1.1);
  \draw[] (-.45,1) node {$\scriptstyle \dots$} (.55,1) node {$\scriptstyle \dots$};
  \draw[] (-.45,-.35) node {$\scriptstyle \dots$} (.55,-.35) node {$\scriptstyle \dots$};
  \draw[onearrow] (delta2) -- (delta1);
    \draw[decorate,decoration=brace] (-.9,1.15) -- node[auto] {$\scriptstyle \mathbf n_{2}$} (-.1,1.15);
    \draw[decorate,decoration=brace] (.1,1.15) -- node[auto] {$\scriptstyle \mathbf n_{1}$} (.9,1.15);
    \draw[decorate,decoration=brace] (-.2,-.55) -- node[auto] {$\scriptstyle \mathbf m_{2}$} (-.8,-.55);
    \draw[decorate,decoration=brace] (.8,-.55) -- node[auto] {$\scriptstyle \mathbf m_{1}$} (.2,-.55);
\end{tikzpicture}
:
   P\bigl(\mathbf m_{1}, \{*\} \sqcup\mathbf n_{1}\bigr) \otimes P\bigl(\mathbf m_{2} \sqcup \{*\},\mathbf n_{2}\bigr) \to P\bigl(\mathbf m_{1} \sqcup \mathbf m_{2}, \mathbf n_{1} \sqcup \mathbf n_{2} \bigr) $$
   There are associativity axioms for each of the following types of diagrams:
  $$
  
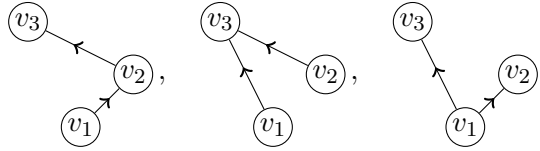
\begin{tikzpicture}[baseline=(v2.base)]
    \path (0,0) node[draw,circle,inner sep = 1pt] (v1) {$v_{1}$};
    \path (.7,.7) node[draw,circle,inner sep = 1pt] (v2) {$v_{2}$};
    \path (-.7,1.4) node[draw,circle,inner sep = 1pt] (v3) {$v_{3}$};
    \draw[,onearrow] (v1) -- (v2);
    \draw[,onearrow] (v2) -- (v3);
  \end{tikzpicture}\,,\quad
  \begin{tikzpicture}[baseline=(v2.base)]
    \path (0,0) node[draw,circle,inner sep = 1pt] (v1) {$v_{1}$};
    \path (.7,.7) node[draw,circle,inner sep = 1pt] (v2) {$v_{2}$};
    \path (-.7,1.4) node[draw,circle,inner sep = 1pt] (v3) {$v_{3}$};
    \draw[,onearrow] (v1) -- (v3);
    \draw[,onearrow] (v2) -- (v3);
  \end{tikzpicture}\,,\quad
  \begin{tikzpicture}[baseline=(v2.base)]
    \path (0,0) node[draw,circle,inner sep = 1pt] (v1) {$v_{1}$};
    \path (.7,.7) node[draw,circle,inner sep = 1pt] (v2) {$v_{2}$};
    \path (-.7,1.4) node[draw,circle,inner sep = 1pt] (v3) {$v_{3}$};
    \draw[,onearrow] (v1) -- (v2);
    \draw[,onearrow] (v1) -- (v3);
  \end{tikzpicture}
  $$
  
  A \define{coproperad} has instead a \define{decomposition} map $ P\bigl(\mathbf m_{1} \sqcup \mathbf m_{2}, \mathbf n_{1} \sqcup \mathbf n_{2} \bigr) \to P\bigl(\mathbf m_{1},\mathbf k \sqcup \mathbf n_{1} \bigr) \otimes P\bigl(\mathbf m_{2} \sqcup \mathbf k,\mathbf n_{2}\bigr) $ for each tuple $(\mathbf m_{1},\mathbf m_{2},\mathbf n_{2},\mathbf n_{2},\mathbf k)$, satisfying coassociativity axioms.  A \define{codioperad} similarly has decomposition maps whenever $\mathbf k = \{*\}$.
\end{definition}

Our convention will be that (co)(di/pr)operads may be non(co)unital.

\begin{definition}\label{defn.modelcatstr}
  For any $V \in \DGVect$, the (di/pr)operad $\End(V)$ satisfies $\End(V)(\mathbf m,\mathbf n) = \hom(V^{\otimes \mathbf m},V^{\otimes \mathbf n})$.  An \define{action} of a (di/pr)operad $P$ on $V$ is a homomorphism $P \to \End(V)$.  If $V$ is equipped with an action of $P$, then we will call $V$ a \define{$P$-algebra}.
  
  The category of (di/pr)operads has a model category structure in which the weak equivalences are the quasiisomorphisms, and the fibrations are the surjections~\cite[Appendix~A]{MR2572248}.  
  Abstract nonsense of model categories guarantees that if $\h_{1}P$ and $\h_{2}P$ are any two cofibrant replacements of the same (di/pr)operad $P$, then we can turn any action of $\h_{1}P$ on $V$ into an action of $\h_{2}P$, and vice versa, and the spaces of choices required to do so are contractible.  Thus we are justified in saying that a \define{homotopy $P$-action on $V$} is an action on $V$ of any cofibrant replacement $\h P$ of $P$.
\end{definition}

There is a forgetful functor from properads to dioperads, whose left adjoint defines the \define{universal enveloping properad} of a dioperad.  The reader should be warned that these functors are known not to be exact~\cite[Theorem~47]{MR2560406}.
For comparison, define a \define{prop} to be an $\SS$-bimodule in which composition is defined even when $\mathbf k = \emptyset$ (and satisfying some extra commutativity/associativity axioms for disconnected  graphs).  By~\cite{MR2320654}, the forgetful functor from props to properads and its adjoint constructing the universal enveloping prop of a properad are exact.  In particular, properadic homotopy actions always extend to propic homotopy actions.

We now list the (di/pr)operads that will be of primary interest:

\begin{definition} \label{defn.qlocproperad}
  It follows from Lemma~\ref{lemma.cochains} and the triangle inequality that $\qloc$ and $\qloc^{\inv}$ are a sub-(di/pr)operads of both $\End(\C_{\bullet}(\RR))$ and $\End(\C^{1-\bullet}(\RR))$.  An action of $P$ on $\C_{\bullet}(\RR)$ or $\C^{1-\bullet}(\RR)$ is \define{quasilocal} if it factors through $\qloc$, and \define{quasilocal and translation invariant} if it factors through $\qloc^{\inv}$.
\end{definition}

\begin{definition}\label{defn.frob1}
  The (di/pr)operad $\Frob_{1}$ of \define{open and coopen $1$-shifted commutative Frobenius algebras} satisfies:
  $$ \dim \Frob_{1}(m,n)_{\bullet} = \begin{cases} 1, & m,n > 0,\ (m,n)\neq (1,1), \text{ and } \bullet = 1-m, \\ 0, & mn \leq 1 \text{ or } \bullet \neq 1-m.\end{cases} $$
  The $\SS_{n}$ action on $\Frob_{1}(m,n)$ is trivial, whereas $\SS_{m}$ acts via the sign representation.  The composition $P\bigl(\mathbf m_{1},\mathbf k \sqcup \mathbf n_{1} \bigr) \otimes P\bigl(\mathbf m_{2} \sqcup \mathbf k,\mathbf n_{2}\bigr) \to P\bigl(\mathbf m_{1} \sqcup \mathbf m_{2}, \mathbf n_{1} \sqcup \mathbf n_{2} \bigr)$ is $0$ unless $k = \{*\}$, in which case it is multiplication by $(-1)^{m_{2}}$.
\end{definition}

  Implicit in the description of $\Frob_{1}$ given in Definition~\ref{defn.frob1} is a choice of basis vector for each vector space $\Frob_{1}(m,n)_{1-m}$.  It is convenient to denote this basis vector as an unmarked vertex with $m$ inputs and $n$ outputs:
  $$ \Frob_{1}(m,n)_{1-m} = \operatorname{span}_{\QQ} \left\{ \tikz[baseline=(basepoint)] {
    \path (0,0) coordinate (basepoint) (0,4pt) node[dot] {} 
    (0,-7pt) node{$\scriptstyle \dots$} (0,15pt) node{$\scriptstyle \dots$};
    \draw[onearrow] (-8pt,-8pt) -- (0,4pt);
    \draw[onearrow] (8pt,-8pt) -- (0,4pt);
    \draw[onearrow] (0,4pt) -- (-8pt,16pt);
    \draw[onearrow] (0,4pt) -- (8pt,16pt);
    \draw[decorate,decoration=brace] (-9pt,18pt) -- node[auto] {$\scriptstyle n$} (9pt,18pt);
    \draw[decorate,decoration=brace] (9pt,-10pt) -- node[auto] {$\scriptstyle m$} (-9pt,-10pt);
   } \right\}$$
   The sign rules given describe compositions shaped as in Definition~\ref{defn.diproperad}:
   $$
 \begin{tikzpicture}[baseline=(basepoint),yscale=1.25]
  \path(0,.1) coordinate (basepoint);
  \draw (-.5,.5) node[dot] (delta1) {};
  \draw (.5,0) node[dot] (delta2) {};
  \draw[onearrow]  (-.25,-.5) .. controls +(0,.25) and +(.25,-.5) .. (delta1);
  \draw[onearrow] (-.75,-.5) .. controls +(0,.25) and +(-.25,-.25) .. (delta1);
  \draw[onearrow]  (.25,-.5) .. controls +(0,.25) and +(-.25,-.25) .. (delta2);
  \draw[onearrow] (.75,-.5) .. controls +(0,.25) and +(.25,-.25) .. (delta2);
  \draw[onearrow] (delta1) .. controls +(-.25,.25) and +(0,-.25) .. (-.85,1.1);
  \draw[onearrow] (delta1) .. controls +(.25,.25) and +(0,-.25) .. (-.15,1.1);
  \draw[onearrow] (delta2) .. controls +(-.25,.5) and +(0,-.25) .. (.15,1.1);
  \draw[onearrow] (delta2) .. controls +(.25,.25) and +(0,-.25) .. (.85,1.1);
  \draw[] (-.45,1) node {$\scriptstyle \dots$} (.55,1) node {$\scriptstyle \dots$};
  \draw[] (-.45,-.35) node {$\scriptstyle \dots$} (.55,-.35) node {$\scriptstyle \dots$};
  \draw[onearrow] (delta2) -- (delta1);
    \draw[decorate,decoration=brace] (-.9,1.15) -- node[auto] {$\scriptstyle n_{2}$} (-.1,1.15);
    \draw[decorate,decoration=brace] (.1,1.15) -- node[auto] {$\scriptstyle n_{1}$} (.9,1.15);
    \draw[decorate,decoration=brace] (-.2,-.55) -- node[auto] {$\scriptstyle m_{2}$} (-.8,-.55);
    \draw[decorate,decoration=brace] (.8,-.55) -- node[auto] {$\scriptstyle m_{1}$} (.2,-.55);
 \end{tikzpicture}
   = (-1)^{m_{2}}
   \tikz[baseline=(basepoint)] {
    \path (0,0) coordinate (basepoint) (0,4pt) node[dot] {} 
    (0,-7pt) node{$\scriptstyle \dots$} (0,15pt) node{$\scriptstyle \dots$};
    \draw[onearrow] (-8pt,-8pt) -- (0,4pt);
    \draw[onearrow] (8pt,-8pt) -- (0,4pt);
    \draw[onearrow] (0,4pt) -- (-8pt,16pt);
    \draw[onearrow] (0,4pt) -- (8pt,16pt);
    \draw[decorate,decoration=brace] (-9pt,18pt) -- node[auto] {$\scriptstyle n_{1}+n_{2}$} (9pt,18pt);
    \draw[decorate,decoration=brace] (9pt,-10pt) -- node[auto] {$\scriptstyle m_{1}+m_{2}$} (-9pt,-10pt);
   }
   $$
   Other compositions can be computed using the symmetric group actions.  Permutations of the outgoing strands act trivially, but permutations of the incoming strands lead to signs.  For example, tracking the actions of both $\SS_{m_{2}+1}$ on the top vertex and $\SS_{m_{1}+m_{2}}$ on the resulting composition gives:
   $$
 \begin{tikzpicture}[baseline=(basepoint),yscale=1.25]
  \path(0,.1) coordinate (basepoint);
  \draw (.5,.5) node[dot] (delta1) {};
  \draw (-.5,0) node[dot] (delta2) {};
  \draw[onearrow]  (-.25,-.5) .. controls +(0,.25) and +(.25,-.5) .. (delta1);
  \draw[onearrow] (-.75,-.5) .. controls +(0,.25) and +(-.25,-.25) .. (delta1);
  \draw[onearrow]  (.25,-.5) .. controls +(0,.25) and +(-.25,-.25) .. (delta2);
  \draw[onearrow] (.75,-.5) .. controls +(0,.25) and +(.25,-.25) .. (delta2);
  \draw[onearrow] (delta1) .. controls +(-.25,.25) and +(0,-.25) .. (-.85,1.1);
  \draw[onearrow] (delta1) .. controls +(.25,.25) and +(0,-.25) .. (-.15,1.1);
  \draw[onearrow] (delta2) .. controls +(-.25,.5) and +(0,-.25) .. (.15,1.1);
  \draw[onearrow] (delta2) .. controls +(.25,.25) and +(0,-.25) .. (.85,1.1);
  \draw[] (-.45,1) node {$\scriptstyle \dots$} (.55,1) node {$\scriptstyle \dots$};
  \draw[] (-.45,-.35) node {$\scriptstyle \dots$} (.55,-.35) node {$\scriptstyle \dots$};
  \draw[onearrow] (delta2) -- (delta1);
    \draw[decorate,decoration=brace] (-.9,1.15) -- node[auto] {$\scriptstyle n_{2}$} (-.1,1.15);
    \draw[decorate,decoration=brace] (.1,1.15) -- node[auto] {$\scriptstyle n_{1}$} (.9,1.15);
    \draw[decorate,decoration=brace] (-.2,-.55) -- node[auto] {$\scriptstyle m_{2}$} (-.8,-.55);
    \draw[decorate,decoration=brace] (.8,-.55) -- node[auto] {$\scriptstyle m_{1}$} (.2,-.55);
 \end{tikzpicture}
   = 
   \tikz[baseline=(basepoint)] {
    \path (0,0) coordinate (basepoint) (0,4pt) node[dot] {} 
    (0,-7pt) node{$\scriptstyle \dots$} (0,15pt) node{$\scriptstyle \dots$};
    \draw[onearrow] (-8pt,-8pt) -- (0,4pt);
    \draw[onearrow] (8pt,-8pt) -- (0,4pt);
    \draw[onearrow] (0,4pt) -- (-8pt,16pt);
    \draw[onearrow] (0,4pt) -- (8pt,16pt);
    \draw[decorate,decoration=brace] (-9pt,18pt) -- node[auto] {$\scriptstyle n_{1}+n_{2}$} (9pt,18pt);
    \draw[decorate,decoration=brace] (9pt,-10pt) -- node[auto] {$\scriptstyle m_{1}+m_{2}$} (-9pt,-10pt);
   }
   \quad \text{ and } \quad
 \begin{tikzpicture}[baseline=(basepoint),yscale=1.25]
  \path(0,.1) coordinate (basepoint);
  \draw (.5,.5) node[dot] (delta1) {};
  \draw (-.5,0) node[dot] (delta2) {};
  \draw[onearrow]  (.25,-.5) .. controls +(0,.25) and +(-.25,-.5) .. (delta1);
  \draw[onearrow] (.75,-.5) .. controls +(0,.25) and +(.25,-.25) .. (delta1);
  \draw[onearrow]  (-.25,-.5) .. controls +(0,.25) and +(.25,-.25) .. (delta2);
  \draw[onearrow] (-.75,-.5) .. controls +(0,.25) and +(-.25,-.25) .. (delta2);
  \draw[onearrow] (delta1) .. controls +(.25,.25) and +(0,-.25) .. (.85,1.1);
  \draw[onearrow] (delta1) .. controls +(-.25,.25) and +(0,-.25) .. (.15,1.1);
  \draw[onearrow] (delta2) .. controls +(.25,.5) and +(0,-.25) .. (-.15,1.1);
  \draw[onearrow] (delta2) .. controls +(-.25,.25) and +(0,-.25) .. (-.85,1.1);
  \draw[] (-.45,1) node {$\scriptstyle \dots$} (.55,1) node {$\scriptstyle \dots$};
  \draw[] (-.45,-.35) node {$\scriptstyle \dots$} (.55,-.35) node {$\scriptstyle \dots$};
  \draw[onearrow] (delta2) -- (delta1);
    \draw[decorate,decoration=brace] (-.9,1.15) -- node[auto] {$\scriptstyle n_{1}$} (-.1,1.15);
    \draw[decorate,decoration=brace] (.1,1.15) -- node[auto] {$\scriptstyle n_{2}$} (.9,1.15);
    \draw[decorate,decoration=brace] (-.2,-.55) -- node[auto] {$\scriptstyle m_{1}$} (-.8,-.55);
    \draw[decorate,decoration=brace] (.8,-.55) -- node[auto] {$\scriptstyle m_{2}$} (.2,-.55);
 \end{tikzpicture}
   = 
   (-1)^{m_{1}m_{2}}
   \tikz[baseline=(basepoint)] {
    \path (0,0) coordinate (basepoint) (0,4pt) node[dot] {} 
    (0,-7pt) node{$\scriptstyle \dots$} (0,15pt) node{$\scriptstyle \dots$};
    \draw[onearrow] (-8pt,-8pt) -- (0,4pt);
    \draw[onearrow] (8pt,-8pt) -- (0,4pt);
    \draw[onearrow] (0,4pt) -- (-8pt,16pt);
    \draw[onearrow] (0,4pt) -- (8pt,16pt);
    \draw[decorate,decoration=brace] (-9pt,18pt) -- node[auto] {$\scriptstyle n_{1}+n_{2}$} (9pt,18pt);
    \draw[decorate,decoration=brace] (9pt,-10pt) -- node[auto] {$\scriptstyle m_{1}+m_{2}$} (-9pt,-10pt);
   }   
   $$

Definitions~\ref{defn.qlocproperad} and~\ref{defn.frob1} are closely related:
\begin{lemma}\label{lemma.hqlocfrob}
  $\H_{\bullet}(\qloc) = \Frob_{1}$.
\end{lemma}

This allows us to introduce the following notion, which will be useful in Section~\ref{section.mainthms}:
\begin{definition}\label{defn.Thom}
  A \define{Thom form} is a cycle in $\qloc$ representing in homology a basis vector of $\Frob_{1}$, for the basis picked out in Definition~\ref{defn.frob1}.
\end{definition}

The justification for the name ``Thom form'' comes from the proofs of Proposition~\ref{homologyofqloc} and Lemma~\ref{lemma.hqlocfrob}: the Thom forms are precisely the forms implementing a Thom isomorphism for the embedding $\diag: \RR \mono \RR^{m+n}$.  There is a sign ambiguity coming from choices of how to work with forms on a product of manifolds and how to treat basis elements under homological-degree reindexing; Definition~\ref{defn.Thom} makes a particular choice for these signs.

\begin{proof}[Proof of Lemma~\ref{lemma.hqlocfrob}]
  The homology of any dg (di/pr)operad is itself a dg (di/pr)operad (but with zero differential).  Proposition~\ref{homologyofqloc} identifies the dimensions of $\H_{\bullet}(\qloc)$ with those of $\Frob_{1}$.  To check that these agree as $\SS$-bimodules requires understanding with what signs a cochain $f$ supported in $B_{\ell}(\diag(\RR))\subseteq \RR^{m+n}$ representing the non-zero homology class in $\qloc(m,n)$ transforms, when thought of as a map $\C_{\bullet}(\RR^{m}) \to \C_{\bullet}(\RR^{n})$.  As in the proof of Proposition~\ref{homologyofqloc}, it is perhaps easiest to think in the familiar de Rham model.  Then $f$ is  represented by a de Rham form of form degree $m+n-1$, which is homologous to the tensor product of a top form on $\RR^{m+n-1} = \diag(\RR)^{\perp}$ with total integral $1$ and the constant function $1$ on $\diag(\RR)$.  The top form transforms, modulo exact forms, in the sign representation of $\SS_{m} \times \SS_{n}$ acting by permuting the coordinates on $\RR^{m+n}$.  However, we care about the permutation action not on $(\C_{\bullet}(\RR^{m}))^{\vee}\otimes \C^{-\bullet}(\RR^{n}) = \bigl((\C_{\bullet}(\RR))^{\vee}\bigr)^{\otimes m}\otimes \bigl(\C^{-\bullet}(\RR)\bigr)^{\otimes n}$, but rather on 
  $$ (\C_{\bullet}(\RR^{m}))^{\vee}\otimes \C_{\bullet}(\RR^{n}) = \bigl((\C_{\bullet}(\RR))^{\vee}\bigr)^{\otimes m}\otimes \bigl(\C_{\bullet}(\RR)\bigr)^{\otimes n} \subseteq \bigl((\C_{\bullet}(\RR))^{\vee}\bigr)^{\otimes m}\otimes \bigl(\C^{-\bullet}(\RR)[1]\bigr)^{\otimes n} $$
  The extra factor of $[1]^{\otimes n}$ switches the sign representation of $\SS_{n}$ into the trivial representation.  This verifies the isomorphism $\H_{\bullet}(\qloc) \cong\Frob_{1}$ as $\SS$-bimodules.
  
  It remains to verify that $\H_{\bullet}(\qloc)$ and $\Frob_{1}$ have the same (di/pr)operadic compositions.  Binary properadic composition along graphs with $k\geq 2$ internal edges vanishes in $\H_{\bullet}(\qloc)$ for homological degree reasons: there are no non-zero degree-zero maps $[-m_{1}+1]\otimes[-m_{2}+1] \to [-(m_{1}+m_{2}-k)+1]$ unless $k=1$.  It therefore suffices to show that all dioperadic compositions in $\H_{\bullet}(\qloc)$ are non-zero.  The exact coefficients will depend on the choice of basis for $\H_{\bullet}(\qloc)$.  Up to sign, this basis can be chosen by declaring that the basis vector is the fundamental class of $\diag(\RR)$, but fixing the signs requires deciding in more detail the sign conventions for how this class should behave under homological degree reindexing.
  
  Consider then a dioperadic composition in $\H_{\bullet}(\qloc)$ of the form 
    $$
  
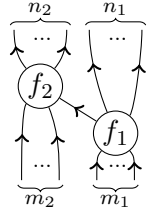
\begin{tikzpicture}[baseline=(basepoint),yscale=1.25]
  \path(0,.25) coordinate (basepoint);
  \draw (-.5,.5) node[circle,draw,inner sep=1] (delta1) {$f_{2}$};
  \draw (.5,0) node[circle,draw,inner sep=1] (delta2) {$f_{1}$};
  \draw[onearrow]  (-.25,-.5) .. controls +(0,.25) and +(.25,-.5) .. (delta1);
  \draw[onearrow] (-.75,-.5) .. controls +(0,.25) and +(-.25,-.25) .. (delta1);
  \draw[onearrow]  (.25,-.5) .. controls +(0,.25) and +(-.25,-.25) .. (delta2);
  \draw[onearrow] (.75,-.5) .. controls +(0,.25) and +(.25,-.25) .. (delta2);
  \draw[onearrow] (delta1) .. controls +(-.25,.25) and +(0,-.25) .. (-.85,1.1);
  \draw[onearrow] (delta1) .. controls +(.25,.25) and +(0,-.25) .. (-.15,1.1);
  \draw[onearrow] (delta2) .. controls +(-.25,.5) and +(0,-.25) .. (.15,1.1);
  \draw[onearrow] (delta2) .. controls +(.25,.25) and +(0,-.25) .. (.85,1.1);
  \draw[] (-.45,1) node {$\scriptstyle \dots$} (.55,1) node {$\scriptstyle \dots$};
  \draw[] (-.45,-.35) node {$\scriptstyle \dots$} (.55,-.35) node {$\scriptstyle \dots$};
  \draw[onearrow] (delta2) -- (delta1);
    \draw[decorate,decoration=brace] (-.9,1.15) -- node[auto] {$\scriptstyle  n_{2}$} (-.1,1.15);
    \draw[decorate,decoration=brace] (.1,1.15) -- node[auto] {$\scriptstyle  n_{1}$} (.9,1.15);
    \draw[decorate,decoration=brace] (-.2,-.55) -- node[auto] {$\scriptstyle  m_{2}$} (-.8,-.55);
    \draw[decorate,decoration=brace] (.8,-.55) -- node[auto] {$\scriptstyle  m_{1}$} (.2,-.55);
\end{tikzpicture}
    $$
  where $f_{1}$ and $f_{2}$ are cycles in $\qloc$ representing non-zero homology classes.  We will continue to think in terms of the de Rham model.  Choose a compactly supported smooth 1-form $\varphi$ on $\RR$ such that $\int_{\RR}\varphi = 1$.  Denote the standard coordinates on $\RR^{m_{1}+n_{1}+1}$ by $x_{1},\dots,x_{m_{1}+n_{1}+1}$ and the standard coordinates on $\RR^{m_{2}+n_{2}+1}$ by $y_{1},\dots,y_{m_{2}+n_{2}+1}$, and suppose that the dioperadic composition pairs the last coordinate $x_{m_{1}+n_{1}+1}$ with the first coordinate $y_{1}$.  Up to a nonzero scalar factor and the addition of exact forms, we may suppose:
  \begin{gather*}
    f_{1}(x_{1},\dots,x_{m_{1}+n_{1}+1}) = \varphi(x_{2}-x_{1})\varphi(x_{3}-x_{2})\dots\varphi(x_{m_{1}+n_{1}+1} - x_{m_{1}+n_{1}})\\
       f_{2}(y_{1},\dots,y_{m_{2}+n_{2}+1}) = \varphi(y_{2}-y_{1})\varphi(y_{3}-y_{2})\dots\varphi(y_{m_{1}+n_{1}+1} - y_{m_{1}+n_{1}})
  \end{gather*}
  The composition is computed (up to sign conventions) via the following integral: set $x_{m_{1}+n_{1}+1} = y_{1} = z$, and integrate out this variable.  Thus the dioperadic composition is
  $$
     f_{1}\Circ f_{2} = \varphi(x_{2}-x_{1}) \dots \varphi(x_{m_{1}+n_{1}} - x_{m_{1}+n_{1}-1}) \psi(y_{2} - x_{m_{1}+n_{1}})\varphi(y_{3}-y_{2})\dots\varphi(y_{m_{1}+n_{1}+1} - y_{m_{1}+n_{1}}),
   $$
   where $\psi(y-x) = \int_{z \in \RR} \varphi(y-z)\varphi(z-x)$.  Note that indeed $\psi(y-x)$  depends only on the difference $y-x$, and not on their individual values, and $\psi$ is a compactly supported 1-form.  But one may also easily compute $\int_{\RR}\psi = \int_{t\in \RR}\psi(t) = \int_{t\in \RR}\int_{z\in \RR} \varphi(t-z)\varphi(z) = \int_{z\in \RR}\left(\int_{t\in \RR} \varphi(t-z)\right)\varphi(z) = \int_{z\in \RR}\varphi(z) = 1$.  It follows that $f_{1}\Circ f_{2}$ represents a nonzero homology class in $\qloc(m_{1}+m_{2},n_{1}+n_{2})$.
\end{proof}

The universal enveloping properad of the dioperad $\Frob_{1}$ is the properad $\Frob_{1}$: the fact that composition vanishes whenever $k = |\mathbf k| > 0$ follows from the $\SS$-actions.  Thus the notion of ``$\Frob_{1}$-algebra'' is unambiguous.  However, because the universal enveloping properad functor is not exact, the notion of ``homotopy $\Frob_{1}$-algebra'' is ambiguous.  Namely, let $\h^{\di}\Frob_{1}$ denote a cofibrant replacement of $\Frob_{1}$ in the category of dioperads, and $\h^{\pr}\Frob_{1}$ a cofibrant replacement in properads.  Since forgetting is not exact, the underlying dioperad of $\h^{\pr}\Frob_{1}$ may not be cofibrant, although it does fiber acyclicly over $\Frob_{1}$; the universal enveloping properad of $\h^{\di}\Frob_{1}$ will be cofibrant, but may not fiber acyclicly over $\Frob_{1}$.  There is a canonical contractible space of homomorphisms $\h^{\di}\Frob_{1} \to \h^{\pr}\Frob_{1}$, through which any $\h^{\pr}\Frob_{1}$-algebra forgets to an $\h^{\di}\Frob_{1}$-algebra.  As we will see in Section~\ref{section.mainthms} from trying to represent each on $\C_{\bullet}(\RR)$, this forgetting really is a loss of information: not all actions of $\h^{\di}\Frob_{1}$ extend to actions of $\h^{\pr}\Frob_{1}$.

\section{Koszulity of $\Frob_{1}$} \label{section.Koszul}

As far as this paper is concerned, the \emph{raison d'\^etre} of Koszul duality theory is to provide small cofibrant replacements of objects of interest.
We will briefly recall enough of the theory for our purposes.

  For any $\SS$-bimodule $T$, we denote by $\Ff(T)$ the free (di/pr)operad generated by $T$.  (This is not much of an abuse of notation, as the universal enveloping properad of the free dioperad generated by $T$ is the free properad generated by $T$.)  Note that $\Ff(T)$ is also a co(di/pr)operad.  We let $\Ff^{(k)}(T)$ denote the sub-$\SS$-bimodule of $\Ff(T)$ that transforms with weight $k$ under the canonical $\QQ^{\times}$-action on $T$. In particular, $\Ff^{(1)}(T) = T$, and $\Ff(T) = \bigoplus_{k\geq 1}\Ff^{(k)}(T)$.  
    
\begin{definition} \label{defn.quaddual}
  A \define{quadratic} (di/pr)operad is a (di/pr)operad presented as $P = \Ff(T) / \<R\>$, where $T$ is an $\SS$-bimodule and $R \subseteq \Ff^{(2)}(T)$ is a sub-$\SS$-bimodule generating the ideal $\<R\>$.
  
  The \define{quadratic dual} $P^{\shriek}$ of a quadratic (di/pr)operad $P$ is the maximal graded sub-co(di/pr)operad of $\Ff(T[1])$ whose intersection with $\Ff^{(2)}(T[1])$ is precisely $R[2]\subseteq \Ff^{(2)}(T[1])$.
\end{definition}

\begin{definition}
  Let $Q$ be any co(di/pr)operad, such that for each $m_{1},m_{2},n_{1},n_{2}$, there are only finitely many $k\in \NN$ for which the decomposition map $ Q\bigl( m_{1} \sqcup  m_{2},  n_{1} \sqcup  n_{2} \bigr) \to Q\bigl( m_{1}, k \sqcup  n_{1} \bigr) \otimes Q\bigl( m_{2} \sqcup  k, n_{2}\bigr) $ is nonzero.
  The \define{cobar construction} applied to $Q$ produces the (di/pr)operad $\B Q = \Ff(Q[-1])$, with the differential extending the degree-$(-1)$ map $\sum (\text{decompositions}) : Q[-1] \to \Ff^{(2)}(Q[-1])$.  Coassiciativity is equivalent to the differential squaring to $0$.
\end{definition}

One of the many uses of the cobar construction is to organize the inductive computation of quadratic duals:

\begin{lemma} \label{lemma.cobarinductive}
  Let $P = \Ff(T)/\<R\>$ be a quadratic (di/pr)operad and $P^{\shriek}$ its quadratic dual.  The canonical $\QQ^{\times}$ action on $T$ extends to a grading by positive integers of both $P$ and $P^{\shriek}$.  For $x\in \Ff^{(k)}(T[1])$, let $\tilde x \in \Ff^{(k)}(T[1])[-1]$ denote the corresponding generator of $\B\Ff(T[1])$, and $\partial(\tilde x)$ its differential therein.  Suppose that $k\geq 3$.  Then $x\in P^{\shriek}$ if and only if  $\partial(\tilde x) \in \B P^{\shriek} \subseteq \B\Ff(T[1])$.
\end{lemma}

\begin{proof}
As in more familiar cases, a co(di/pr)operad $Q$ is entirely determined by its cobar construction $\B Q$ along with the inclusion $Q[-1] \mono \B Q$ of (non-dg) $\SS$-bimodules~\cite{MR2320654}.  In particular, the different binary decompositions of an element $x\in Q$ correspond (without cancelations) to the summands appearing in $\partial(\tilde x)$, where $\tilde x \in Q[-1]$ is the corresponding element of the shifted $\SS$-bimodule.

Suppose that $x\in \Ff^{(k)}(T[1])$ with $k\geq 3$.  Definition~\ref{defn.quaddual} then says that $x\in P^{\shriek}$ if and only if all binary decompositions of $x$ are elements of $\Ff^{(2)}(P^{\shriek})$.  (The binary decompositions of $x$ consist of elements of individual weights, for the $\QQ^{\times}$ action on $T$, each less than $k$, and so whether or not they are in $P^{\shriek}$ has been determined by induction; in weights $1$ and $2$ one has $T[1]$ and $R[2]$ respectively.
)   But the general yoga of the cobar construction says that this happens if and only if $\partial(\tilde x) \in \B(P^{\shriek})$.
\end{proof}

The (di/pr)operad $\B Q$ is an example of a \define{quasifree} (di/pr)operad, which more generally is any dg (di/pr)operad which would be free if one were to forget its differential.

\begin{lemma}\label{lemma.kos}
  Let $P$ be a  quadratic (di/pr)operad, with generators $T$ and relations $R$. Then $\B P^{\shriek}$ is cofibrant and fibers over $P$.
\end{lemma}
\begin{proof}
  Cofibrancy follows from~\cite[Corollary~40]{MR2572248}.  To define the map $\B P^{\shriek} \to P$, it suffices to define the action on generators $P^{\shriek}[-1] \subseteq \Ff(T[1])[-1]$.  We declare that $T[1][-1] \subseteq P^{\shriek}[-1]$ maps by the identity to $T \subseteq P$, and that all other generators map to $0$.  To check that this is well-defined, it suffices to check that the derivatives of the generators in $\Ff^{(2)}(T[1])[-1]$ get mapped to $0$.  But these generators are precisely a copy of $R[1]$, and differentiating and mapping gives the (trivial) image of $R$ in $P = \Ff(T)/\<R\>$.
\end{proof}

\begin{definition}
  A quadratic (di/pr)operad is \define{Koszul} if the canonical fibration $\B P^{\shriek} \to P$ from Lemma~\ref{lemma.kos} is acyclic, in which case $\B P^{\shriek}$ is a cofibrant replacement of $P$. 
  When $P$ is Koszul, we let $\sh P = \B P^{\shriek}$,  and call $\sh P$-algebras \define{strong homotopy $P$-algebras}.
\end{definition}

For any (di/pr)operad $P$ satisfying some mild finite-dimensionality assumptions, the (di/pr)operad $\B((\B P^{*})^{*})$ is always a cofibrant replacement of $P$.  (The second dual should be taken relative to the grading induced by the $\QQ^{\times}$ action on the $\SS$-bimodule $P$.)  The point is that $\B P^{\shriek}$ is generally much smaller than $\B((\B P^{*})^{*})$, and hence more manageable.

The main result of this section says that $\Frob_{1}$ is Koszul, as both a dioperad and as a properad:

\begin{proposition}
  The (di/pr)operad $\Frob_{1}$ of open and coopen commutative Frobenius algebras has the following quadratic presentation, with respect to which it is Koszul.  The generating $\SS$-bimodule $T$ is spanned by:
    $$\underbrace{
  \,\tikz[baseline=(basepoint)]{ 
    \path (0,0) coordinate (basepoint) (0,4pt) node[dot] {};
    \draw[onearrow](0,-6pt) -- (0,4pt);
    \draw[onearrow](0,4pt) -- (-8pt,14pt);
    \draw[onearrow](0,4pt) -- (8pt,14pt);
  }\,
  = 
  \,\tikz[baseline=(basepoint)]{ 
    \path (0,0) coordinate (basepoint) (0,4pt) node[dot] {};
    \draw[onearrow](0,-6pt) -- (0,4pt);
    \draw[](0,4pt) .. controls +(8pt,4pt) and +(8pt,-4pt) .. (-8pt,14pt);
    \draw[](0,4pt) .. controls +(-8pt,4pt) and +(-8pt,-4pt) .. (8pt,14pt);
  }\,
  ,}_{\text{homological degree $0$}} \quad 
  \underbrace{\,\tikz[baseline=(basepoint)]{ 
    \path (0,0) coordinate (basepoint) (0,4pt) node[dot] {};
    \draw[onearrow](-8pt,-6pt) -- (0,4pt);
    \draw[onearrow](8pt,-6pt) -- (0,4pt);
    \draw[onearrow](0,4pt) -- (0,14pt);
  }\,
  = -
  \,\tikz[baseline=(basepoint)]{ 
    \path (0,0) coordinate (basepoint) (0,4pt) node[dot] {};
    \draw[] (-8pt,-6pt) .. controls +(8pt,4pt) and +(8pt,-4pt) .. (0,4pt);
    \draw[] (8pt,-6pt) .. controls +(-8pt,4pt) and +(-8pt,-4pt) .. (0,4pt);
    \draw[onearrow](0,4pt) -- (0,14pt);
  }\,}_{\text{homological degree $-1$}}
  $$
  A basis for the relations $R$ is:
  \begin{gather*}
  \,\tikz[baseline=(basepoint)]{ 
    \path (0,3pt) coordinate (basepoint) (0,1pt) node[dot] {} (-8pt,11pt) node[dot] {};
    \draw[](0,-9pt) -- (0,1pt);
    \draw[](0,1pt) -- (-8pt,11pt);
    \draw[](-8pt,11pt) -- (-12pt,21pt);
    \draw[](-8pt,11pt) -- (0pt,21pt);
    \draw[](0,1pt) -- (12pt,21pt);
  }\,
  -
  \,\tikz[baseline=(basepoint)]{ 
    \path (0,3pt) coordinate (basepoint) (0,1pt) node[dot] {} (8pt,11pt) node[dot] {};
    \draw[](0,-9pt) -- (0,1pt);
    \draw[](0,1pt) -- (8pt,11pt);
    \draw[](0,1pt) -- (-12pt,21pt);
    \draw[](8pt,11pt) -- (0pt,21pt);
    \draw[](8pt,11pt) -- (12pt,21pt);
  }\,
  ,\; 
  \,\tikz[baseline=(basepoint)]{ 
    \path (0,3pt) coordinate (basepoint) (0,1pt) node[dot] {} (-8pt,11pt) node[dot] {};
    \draw[](0,-9pt) -- (0,1pt);
    \draw[](0,1pt) -- (-8pt,11pt);
    \draw[](-8pt,11pt) -- (-12pt,21pt);
    \draw[](-8pt,11pt) -- (0pt,21pt);
    \draw[](0,1pt) -- (12pt,21pt);
  }\,
  -
  \,\tikz[baseline=(basepoint)]{ 
    \path (0,3pt) coordinate (basepoint) (0,1pt) node[dot] {} (8pt,11pt) node[dot] {};
    \draw[](0,-9pt) -- (0,1pt);
    \draw[](0,1pt) -- (8pt,11pt);
    \draw[](0,1pt) -- (0pt,21pt);
    \draw[](8pt,11pt) -- (-12pt,21pt);
    \draw[](8pt,11pt) -- (12pt,21pt);
  }\,
  ,\;
  \,\tikz[baseline=(basepoint)]{ 
    \path (0,3pt) coordinate (basepoint) (8pt,1pt) node[dot] {} (0,11pt) node[dot] {};
    \draw[](-12pt,-9pt) -- (0,11pt);
    \draw[](0pt,-9pt) -- (8pt,1pt);
    \draw[](12pt,-9pt) -- (8pt,1pt);
    \draw[](8pt,1pt) -- (0,11pt);
    \draw[](0,11pt) -- (0,21pt);
  }\,
  +
  \,\tikz[baseline=(basepoint)]{ 
    \path (0,3pt) coordinate (basepoint) (-8pt,1pt) node[dot] {} (0,11pt) node[dot] {};
    \draw[](12pt,-9pt) -- (0,11pt);
    \draw[](-12pt,-9pt) -- (-8pt,1pt);
    \draw[](0pt,-9pt) -- (-8pt,1pt);
    \draw[](-8pt,1pt) -- (0,11pt);
    \draw[](0,11pt) -- (0,21pt);
  }\,
  ,\; 
  \,\tikz[baseline=(basepoint)]{ 
    \path (0,3pt) coordinate (basepoint) (8pt,1pt) node[dot] {} (0,11pt) node[dot] {};
    \draw[](-12pt,-9pt) -- (0,11pt);
    \draw[](0pt,-9pt) -- (8pt,1pt);
    \draw[](12pt,-9pt) -- (8pt,1pt);
    \draw[](8pt,1pt) -- (0,11pt);
    \draw[](0,11pt) -- (0,21pt);
  }\,
  +
  \,\tikz[baseline=(basepoint)]{ 
    \path (0,3pt) coordinate (basepoint) (8pt,1pt) node[dot] {} (0,11pt) node[dot] {};
    \draw[](0pt,-9pt) -- (0,11pt);
    \draw[](12pt,-9pt) -- (8pt,1pt);
    \draw[](-12pt,-9pt) -- (8pt,1pt);
    \draw[](8pt,1pt) -- (0,11pt);
    \draw[](0,11pt) -- (0,21pt);
  }\,
  ,\; 
  \\
  \,\tikz[baseline=(basepoint)]{ 
    \path (0,3pt) coordinate (basepoint) (0,1pt) node[dot] {} (0,11pt) node[dot] {};
    \draw[](-8pt,-9pt) -- (0,1pt);
    \draw[](8pt,-9pt) -- (0,1pt);
    \draw[](0,1pt) -- (0,11pt);
    \draw[](0,11pt) -- (-8pt,21pt);
    \draw[](0,11pt) -- (8pt,21pt);
  }\,
  -
  \,\tikz[baseline=(basepoint)]{ 
    \path (0,3pt) coordinate (basepoint) (8pt,1pt) node[dot] {} (-8pt,11pt) node[dot] {};
    \draw[](-8pt,-9pt) -- (-8pt,11pt);
    \draw[](8pt,-9pt) -- (8pt,1pt);
    \draw[](8pt,1pt) -- (-8pt,11pt);
    \draw[](-8pt,11pt) -- (-8pt,21pt);
    \draw[](8pt,1pt) -- (8pt,21pt);
  }\,
  ,\;
  \,\tikz[baseline=(basepoint)]{ 
    \path (0,3pt) coordinate (basepoint) (0,1pt) node[dot] {} (0,11pt) node[dot] {};
    \draw[](-8pt,-9pt) -- (0,1pt);
    \draw[](8pt,-9pt) -- (0,1pt);
    \draw[](0,1pt) -- (0,11pt);
    \draw[](0,11pt) -- (-8pt,21pt);
    \draw[](0,11pt) -- (8pt,21pt);
  }\,
  -
  \,\tikz[baseline=(basepoint)]{ 
    \path (0,3pt) coordinate (basepoint) (-8pt,1pt) node[dot] {} (8pt,11pt) node[dot] {};
    \draw[](-8pt,-9pt) -- (-8pt,1pt);
    \draw[](8pt,-9pt) -- (8pt,11pt);
    \draw[](-8pt,1pt) -- (8pt,11pt);
    \draw[](-8pt,1pt) -- (-8pt,21pt);
    \draw[](8pt,11pt) -- (8pt,21pt);
  }\,
  ,\;
  \,\tikz[baseline=(basepoint)]{ 
    \path (0,3pt) coordinate (basepoint) (0,1pt) node[dot] {} (0,11pt) node[dot] {};
    \draw[](-8pt,-9pt) -- (0,1pt);
    \draw[](8pt,-9pt) -- (0,1pt);
    \draw[](0,1pt) -- (0,11pt);
    \draw[](0,11pt) -- (-8pt,21pt);
    \draw[](0,11pt) -- (8pt,21pt);
  }\,
  -
  \,\tikz[baseline=(basepoint)]{ 
    \path (0,3pt) coordinate (basepoint) (0,1pt) node[dot] {} (0,11pt) node[dot] {};
    \draw[](0,-9pt) -- (0,1pt);
    \draw[](0,1pt) -- (0,11pt);
    \draw[](0,11pt) -- (0,21pt);
    \draw[](-16pt,-9pt) -- (0,11pt);
    \draw[](0pt,1pt) -- (-16pt,21pt);
  }\,
  ,\;
  \,\tikz[baseline=(basepoint)]{ 
    \path (0,3pt) coordinate (basepoint) (0,1pt) node[dot] {} (0,11pt) node[dot] {};
    \draw[](-8pt,-9pt) -- (0,1pt);
    \draw[](8pt,-9pt) -- (0,1pt);
    \draw[](0,1pt) -- (0,11pt);
    \draw[](0,11pt) -- (-8pt,21pt);
    \draw[](0,11pt) -- (8pt,21pt);
  }\,
  -
  \,\tikz[baseline=(basepoint)]{ 
    \path (0,3pt) coordinate (basepoint) (0,1pt) node[dot] {} (0,11pt) node[dot] {};
    \draw[](0,-9pt) -- (0,1pt);
    \draw[](0,1pt) -- (0,11pt);
    \draw[](0,11pt) -- (0,21pt);
    \draw[](16pt,-9pt) -- (0,11pt);
    \draw[](0pt,1pt) -- (16pt,21pt);
  }\,
  . \end{gather*}
\end{proposition}

\begin{proof}
  The signs arise because the degree-$(-1)$ multiplication does not make a $\Frob_{1}$-algebra $V$ into a commutative algebra, but rather makes $V[-1]$ into a commutative algebra.  It is clear that these define the dioperad $\Frob_{1}$ and hence its universal enveloping properad.  The properad $\Frob_{1}$ has no operations with genus, because the composition $
  \,\tikz[baseline=(basepoint)]{
    \path (0,0pt) coordinate (basepoint) (0,1pt) node[dot] {} (0,7pt) node[dot] {};
    \draw (0,-4pt) -- (0,0pt); \draw (0,8pt) -- (0,12pt);
    \draw (0,1pt) .. controls +(4pt,1pt) and +(4pt,-1pt) .. (0,7pt);
    \draw (0,1pt) .. controls +(-4pt,1pt) and +(-4pt,-1pt) .. (0,7pt);
  }\,$ vanishes, as it must transform both trivially and by the sign representation under the $\SS_{2}$-action interchanging the two interior edges.
  
  The suboperad of $\Frob_{1}$ generated by just the multiplication $\mult$ is a \define{shear} of the nonunital commutative operad (meaning its representations on $V$ are representations of $\mathrm{Com}$ on $V[1]$), which is known to be Koszul~\cite[Theorem~8.5.7]{MR2954392}; similarly the comultiplication $\comult$ generates a copy of the Koszul properad defining nonunital cocommutative coalgebras.  The second line of relations, along with the relation $
  \,\tikz[baseline=(basepoint)]{
    \path (0,0pt) coordinate (basepoint) (0,1pt) node[dot] {} (0,7pt) node[dot] {};
    \draw (0,-4pt) -- (0,0pt); \draw (0,8pt) -- (0,12pt);
    \draw (0,1pt) .. controls +(4pt,1pt) and +(4pt,-1pt) .. (0,7pt);
    \draw (0,1pt) .. controls +(-4pt,1pt) and +(-4pt,-1pt) .. (0,7pt);
  }\,=0$, are together a ``replacement rule'' in the sense of~\cite[8.1]{MR2320654}, and hence $\Frob_{1}$ is Koszul by~\cite[Proposition 8.4]{MR2320654}.
\end{proof}

\begin{corollary}\label{cor.shFrob}
  $\Frob_{1}$ has a cofibrant replacement $\sh\Frob_{1}$, which is quasifree with generating $\SS$-bimodule $(\Frob_{1})^{\shriek}[-1]$.  The generating $\SS$-bimodule $T$ of $\Frob_{1}$ has a bigrading by $\bigl(\#\mult,\#\comult\bigr)$, and the relations $R$ are homogeneous for this bigrading, hence this bigrading extends to $(\Frob_{1})^{\shriek}[-1]$, and the differential on $\sh\Frob_{1}$ preserves this bigrading.
  
  If we are working with dioperads, the piece $(\Frob_{1})^{\shriek}[-1](m,n)$ with $m$ inputs and $n$ outputs is homogeneous for this bigrading, with $\#\mult = m-1$ and $\#\comult = n-1$, and is entirely in homological degree $\#\comult - 1 = n-2$.
  
  If we are working with properads, the piece $(\Frob_{1})^{\shriek}[-1](m,n)$ with $m$ inputs and $n$ outputs has a grading by \define{genus} $\beta$ satisfying
  $ \#\mult = \beta + m-1,$ $ \#\comult = \beta + n-1,$ and $ \text{homological degree} = \#\comult - 1 = \beta + n-2.$ 
\end{corollary}
\begin{proof}
  The formulas follow from definition-unpacking and elementary combinatorics.
\end{proof}

We conclude this section by describing the generators $(\Frob_{1})^{\shriek}[-1]$ of $\sh\Frob_{1}$ for small $\#\mult + \#\comult$.  We will record the representations of the symmetric group using the usual Young diagrams, placed under a generator of that irrep.  The $\#\mult + \#\comult = 1$ piece of $(\Frob_{1})^{\shriek}[-1]$ is a copy of the generators $T$, and decomposes as:
$$
  \underset{\tikz \draw (0,0) rectangle (4pt,4pt); \,\otimes\, \tikz \draw (0,0) rectangle (4pt,4pt) (4pt,0pt) rectangle (8pt,4pt);}{
  \,\tikz[baseline=(main.base)]{
    \node[draw,rectangle,inner sep=2pt] (main) {\comult};
    \draw (main.south) -- ++(0,-8pt);
    \draw (main.north) ++(-6pt,0) -- ++(0,8pt);
    \draw (main.north) ++(6pt,0) -- ++(0,8pt);
  }\,
  }
  \;\oplus\;
\underset{\tikz \draw (0,0) rectangle (4pt,4pt) (0pt,4pt) rectangle (4pt,8pt); \,\otimes\,\tikz \draw (0,0) rectangle (4pt,4pt);}{
  \,\tikz[baseline=(main.base)]{
    \node[draw,rectangle,inner sep=2pt] (main) {\mult};
    \draw (main.north) -- ++(0,8pt);
    \draw (main.south) ++(-6pt,0) -- ++(0,-8pt);
    \draw (main.south) ++(6pt,0) -- ++(0,-8pt);
  }\,
  }
  $$
The $\#\mult + \#\comult = 2$ piece is a copy of $R$, decomposing as:
$$
\underset{\tikz \draw (0,0) rectangle (4pt,4pt); \,\otimes\,  \tikz \draw (0,0) rectangle (4pt,4pt) (4pt,0pt) rectangle (8pt,4pt) (0pt,-4pt)  rectangle  (4pt,0pt) ;  }{
  \,\tikz[baseline=(main.base)]{
    \node[draw,rectangle,inner sep=2pt] (main) {$
        \,\tikz[baseline=(basepoint)]{ 
    \path (0,3pt) coordinate (basepoint) (0,1pt) node[dot] {} (-8pt,11pt) node[dot] {};
    \draw[](0,-9pt) -- (0,1pt);
    \draw[](0,1pt) -- (-8pt,11pt);
    \draw[](-8pt,11pt) -- (-12pt,21pt);
    \draw[](-8pt,11pt) -- (0pt,21pt);
    \draw[](0,1pt) -- (12pt,21pt);
  }\,
  -
  \,\tikz[baseline=(basepoint)]{ 
    \path (0,3pt) coordinate (basepoint) (0,1pt) node[dot] {} (8pt,11pt) node[dot] {};
    \draw[](0,-9pt) -- (0,1pt);
    \draw[](0,1pt) -- (8pt,11pt);
    \draw[](0,1pt) -- (-12pt,21pt);
    \draw[](8pt,11pt) -- (0pt,21pt);
    \draw[](8pt,11pt) -- (12pt,21pt);
  }\,
    $};
    \draw (main.south) -- ++(0,-8pt);
    \draw (main.north) ++(-12pt,0) -- ++(0,8pt);
    \draw (main.north) ++(0pt,0) -- ++(0,8pt);
    \draw (main.north) ++(12pt,0) -- ++(0,8pt);
  }\,
  }
  \;\oplus\;
\underset{\mathbf 2 \,\otimes\,  \mathbf 2  }{
  \,\tikz[baseline=(main.base)]{
    \node[draw,rectangle,inner sep=2pt] (main) {$
          \,\tikz[baseline=(basepoint)]{ 
    \path (0,3pt) coordinate (basepoint) (0,1pt) node[dot] {} (0,11pt) node[dot] {};
    \draw[](-8pt,-9pt) -- (0,1pt);
    \draw[](8pt,-9pt) -- (0,1pt);
    \draw[](0,1pt) -- (0,11pt);
    \draw[](0,11pt) -- (-8pt,21pt);
    \draw[](0,11pt) -- (8pt,21pt);
  }\,
  -
  \,\tikz[baseline=(basepoint)]{ 
    \path (0,3pt) coordinate (basepoint) (8pt,1pt) node[dot] {} (-8pt,11pt) node[dot] {};
    \draw[](-8pt,-9pt) -- (-8pt,11pt);
    \draw[](8pt,-9pt) -- (8pt,1pt);
    \draw[](8pt,1pt) -- (-8pt,11pt);
    \draw[](-8pt,11pt) -- (-8pt,21pt);
    \draw[](8pt,1pt) -- (8pt,21pt);
  }\,
    $};
    \draw (main.south) ++(-6pt,0) -- ++(0pt,-8pt);
    \draw (main.south) ++(6pt,0) -- ++(0pt,-8pt);
    \draw (main.north) ++(-6pt,0) -- ++(0,8pt);
    \draw (main.north) ++(6pt,0) -- ++(0,8pt);
  }\,
  }
  \;\oplus\;
  \underset{ \tikz \draw (0,0) rectangle (4pt,4pt) (4pt,0pt) rectangle (8pt,4pt) (0pt,-4pt)  rectangle  (4pt,0pt) ;  \,\otimes\, \tikz \draw (0,0) rectangle (4pt,4pt);}{
  \,\tikz[baseline=(main.base)]{
    \node[draw,rectangle,inner sep=2pt] (main) {$
      \,\tikz[baseline=(basepoint)]{ 
    \path (0,3pt) coordinate (basepoint) (8pt,1pt) node[dot] {} (0,11pt) node[dot] {};
    \draw[](-12pt,-9pt) -- (0,11pt);
    \draw[](0pt,-9pt) -- (8pt,1pt);
    \draw[](12pt,-9pt) -- (8pt,1pt);
    \draw[](8pt,1pt) -- (0,11pt);
    \draw[](0,11pt) -- (0,21pt);
  }\,
  +
  \,\tikz[baseline=(basepoint)]{ 
    \path (0,3pt) coordinate (basepoint) (-8pt,1pt) node[dot] {} (0,11pt) node[dot] {};
    \draw[](12pt,-9pt) -- (0,11pt);
    \draw[](-12pt,-9pt) -- (-8pt,1pt);
    \draw[](0pt,-9pt) -- (-8pt,1pt);
    \draw[](-8pt,1pt) -- (0,11pt);
    \draw[](0,11pt) -- (0,21pt);
  }\,
    $};
    \draw (main.north) -- ++(0,8pt);
    \draw (main.south) ++(-12pt,0) -- ++(0,-8pt);
    \draw (main.south) ++(0pt,0) -- ++(0,-8pt);
    \draw (main.south) ++(12pt,0) -- ++(0,-8pt);
  }\,
  }
$$
In the middle summand, $\mathbf 2 = \tikz \draw (0,0) rectangle (4pt,4pt) (4pt,0pt) rectangle (8pt,4pt); \oplus \tikz \draw (0,0) rectangle (4pt,4pt) (0pt,4pt) rectangle (4pt,8pt);$ denotes the two-dimensional permutation representation of $\SS_{2}$.  We have adopted the following diagrammatic convention: the diagram inside each rectangle is an element of $(\Frob_{1})^{\shriek} \subseteq \Ff(T[1])$, and the boxed diagram is the corresponding element of $(\Frob_{1})^{\shriek}[-1] \subseteq \B((\Frob_{1})^{\shriek}) \subseteq \B(\Ff(T[1]))$.  Later we will stack diagrams; this denotes composition in $\B((\Frob_{1})^{\shriek}) \subseteq \B(\Ff(T[1]))$.

When $\#\mult + \#\comult \geq 3$, the dioperadic and properadic versions of $(\Frob_{1})^{\shriek}[-1]$ diverge.  The dioperadic version appears as the genus $\beta = 0$ direct summand of the properadic version.  
 Of the  $\#\mult + \#\comult = 3$ part of $(\Frob_{1})^{\shriek}[-1]$, we will compute just the summand with genus $\beta = 1$.  Any composition of $\mult$ and $\comult$ with genus $\beta \geq 1$ has $\#\mult \geq 1$ and $\#\comult \geq 1$.  Since $
  \,\tikz[baseline=(basepoint)]{
    \path (0,0pt) coordinate (basepoint) (0,1pt) node[dot] {} (0,7pt) node[dot] {};
    \draw (0,-4pt) -- (0,0pt); \draw (0,8pt) -- (0,12pt);
    \draw (0,1pt) .. controls +(4pt,1pt) and +(4pt,-1pt) .. (0,7pt);
    \draw (0,1pt) .. controls +(-4pt,1pt) and +(-4pt,-1pt) .. (0,7pt);
  }\,=0$, the only graphs with $\beta = 1$ and $\bigl(\#\mult,\#\comult\bigr) = (1,2)$ are 
\begin{tikzpicture}[baseline=(basepoint)]
    \path (0,5pt) coordinate (basepoint) (0pt,2pt) node[dot] {} (8pt,12pt) node[dot] {} (-4pt,17pt) node[dot] {};
    \draw (0pt,-6pt) -- (0pt,2pt);
    \draw[] (0pt,0pt) -- (8pt,12pt);
    \draw[] (8pt,12pt) -- (-4pt,17pt);
    \draw[] (0pt,0pt) .. controls +(-6pt,6pt) and +(-6pt,-6pt) .. (-4pt,17pt);
    \draw (8pt,12pt) -- (12pt,24pt);
    \draw (-4pt,17pt) -- (-4pt,24pt);
  \end{tikzpicture}
and
\begin{tikzpicture}[baseline=(basepoint)]
    \path (0,5pt) coordinate (basepoint) (0pt,2pt) node[dot] {} (-8pt,12pt) node[dot] {} (4pt,17pt) node[dot] {};
    \draw (0pt,-6pt) -- (0pt,2pt);
    \draw[] (0pt,0pt) -- (-8pt,12pt);
    \draw[] (-8pt,12pt) -- (4pt,17pt);
    \draw[] (0pt,0pt) .. controls +(6pt,6pt) and +(6pt,-6pt) .. (4pt,17pt);
    \draw (-8pt,12pt) -- (-12pt,24pt);
    \draw (4pt,17pt) -- (4pt,24pt);
  \end{tikzpicture}.  We claim that each of these is in $(\Frob_{1})^{\shriek}[-1]$.  Indeed, one may compute that, in $\B(\Ff(T[1]))$, 
  \begin{equation}\label{eqn1}
  \partial\left(
  \,\tikz[baseline=(main.base)]{
    \node[draw,rectangle,inner sep=2pt] (main) {\begin{tikzpicture}[baseline=(basepoint)]
    \path (0,5pt) coordinate (basepoint) (0pt,1pt) node[dot] {} (8pt,12pt) node[dot] {} (-4pt,17pt) node[dot] {};
    \draw (0pt,-6pt) -- (0pt,2pt);
    \draw[] (0pt,0pt) -- (8pt,12pt);
    \draw[] (8pt,12pt) -- (-4pt,17pt);
    \draw[] (0pt,0pt) .. controls +(-6pt,6pt) and +(-6pt,-6pt) .. (-4pt,17pt);
    \draw (8pt,12pt) -- (12pt,24pt);
    \draw (-4pt,17pt) -- (-4pt,24pt);
  \end{tikzpicture}\,};
    \draw (main.south) -- ++(0,-8pt);
    \draw (main.north) ++(-6pt,0) -- ++(0,8pt);
    \draw (main.north) ++(6pt,0) -- ++(0,8pt);
  }\,
  \right) = 
  \;\begin{tikzpicture}[baseline=(outerbasepoint)]
  \node[draw,inner sep=2pt] (box) {
    $
    \,\tikz[baseline=(basepoint)]{ 
      \path (0,3pt) coordinate (basepoint) (0,1pt) node[dot] {} (-8pt,11pt) node[dot] {};
      \draw[](0,-9pt) -- (0,1pt);
      \draw[](0,1pt) -- (-8pt,11pt);
      \draw[](-8pt,11pt) -- (-12pt,21pt);
      \draw[](-8pt,11pt) -- (0pt,21pt);
      \draw[](0,1pt) -- (12pt,21pt);
    }\,
    -
    \,\tikz[baseline=(basepoint)]{ 
      \path (0,3pt) coordinate (basepoint) (0,1pt) node[dot] {} (8pt,11pt) node[dot] {};
      \draw[](0,-9pt) -- (0,1pt);
      \draw[](0,1pt) -- (8pt,11pt);
      \draw[](0,1pt) -- (-12pt,21pt);
      \draw[](8pt,11pt) -- (0pt,21pt);
      \draw[](8pt,11pt) -- (12pt,21pt);
    }\,
    $};
  \draw (box.south) -- ++(0,-10pt) coordinate (bottom);
  \draw (box.north) -- ++(0pt,10pt);
  \draw (box.north) ++(-12pt,0) -- ++(0pt,10pt);
  \draw (box.north) ++(-6pt,18pt) node[draw,inner sep=2pt,fill=white] (upperbox) {\mult};
  \draw (upperbox.north) -- ++(0,10pt) ++(18pt,0) coordinate (top);
  \draw (box.north) ++(12pt,0) -- (top);
  \path (bottom) ++(0,40pt) coordinate (outerbasepoint);  
  \end{tikzpicture}\;
  \;+\;
  \;\begin{tikzpicture}[baseline=(outerbasepoint)]
  \node[draw,inner sep=2pt] (box) {
    $
    \,\tikz[baseline=(basepoint)]{ 
      \path (0,3pt) coordinate (basepoint) (0,1pt) node[dot] {} (0,11pt) node[dot] {};
      \draw[](-8pt,-9pt) -- (0,1pt);
      \draw[](8pt,-9pt) -- (0,1pt);
      \draw[](0,1pt) -- (0,11pt);
      \draw[](0,11pt) -- (-8pt,21pt);
      \draw[](0,11pt) -- (8pt,21pt);
    }\,
    -
    \,\tikz[baseline=(basepoint)]{ 
      \path (0,3pt) coordinate (basepoint) (8pt,1pt) node[dot] {} (-8pt,11pt) node[dot] {};
      \draw[](-8pt,-9pt) -- (-8pt,11pt);
      \draw[](8pt,-9pt) -- (8pt,1pt);
      \draw[](8pt,1pt) -- (-8pt,11pt);
      \draw[](-8pt,11pt) -- (-8pt,21pt);
      \draw[](8pt,1pt) -- (8pt,21pt);
    }\,
    $};
  \draw (box.north) ++(6pt,0) -- ++(0,10pt) coordinate(top);
  \draw (box.north) ++(-6pt,0) -- ++(0,10pt);
  \draw (box.south) ++(6pt,0) -- ++(0,-10pt);
  \draw (box.south) ++(-6pt,0) -- ++(0,-10pt);
  \draw (box.south) ++(0,-18pt) node[draw,inner sep=2pt,fill=white] (lowerbox) {\comult};
  \draw (lowerbox.south) -- ++(0,-10pt) coordinate (bottom);
  \path (bottom) ++(0,40pt) coordinate (outerbasepoint);  
  \end{tikzpicture}\;
  \;\in\; \sh\Frob_{1} = \B\bigl((\Frob_{1})^{\shriek}\bigr),
  \vspace*{-8pt}
  \end{equation}
and therefore $\begin{tikzpicture}[baseline=(basepoint)]
    \path (0,5pt) coordinate (basepoint) (0pt,2pt) node[dot] {} (8pt,12pt) node[dot] {} (-4pt,17pt) node[dot] {};
    \draw (0pt,-6pt) -- (0pt,2pt);
    \draw[] (0pt,0pt) -- (8pt,12pt);
    \draw[] (8pt,12pt) -- (-4pt,17pt);
    \draw[] (0pt,0pt) .. controls +(-6pt,6pt) and +(-6pt,-6pt) .. (-4pt,17pt);
    \draw (8pt,12pt) -- (12pt,24pt);
    \draw (-4pt,17pt) -- (-4pt,24pt);
  \end{tikzpicture} \in (\Frob_{1})^{\shriek}[-1]$ by Lemma~\ref{lemma.cobarinductive}.
It follows that the $\#\mult + \#\comult = 3$ part of the properadic $(\Frob_{1})^{\shriek}[-1]$ is:
$$ \bigl(\beta = 0 \text{ part}\bigr) \;\oplus\;
\underset{\tikz \draw (0,0) rectangle (4pt,4pt); \,\otimes\, \mathbf 2}{
  \,\tikz[baseline=(main.base)]{
    \node[draw,rectangle,inner sep=2pt] (main) {\begin{tikzpicture}[baseline=(basepoint)]
    \path (0,5pt) coordinate (basepoint) (0pt,1pt) node[dot] {} (8pt,12pt) node[dot] {} (-4pt,17pt) node[dot] {};
    \draw (0pt,-6pt) -- (0pt,2pt);
    \draw[] (0pt,0pt) -- (8pt,12pt);
    \draw[] (8pt,12pt) -- (-4pt,17pt);
    \draw[] (0pt,0pt) .. controls +(-6pt,6pt) and +(-6pt,-6pt) .. (-4pt,17pt);
    \draw (8pt,12pt) -- (12pt,24pt);
    \draw (-4pt,17pt) -- (-4pt,24pt);
  \end{tikzpicture}\,};
    \draw (main.south) -- ++(0,-8pt);
    \draw (main.north) ++(-6pt,0) -- ++(0,8pt);
    \draw (main.north) ++(6pt,0) -- ++(0,8pt);
  }\,
}
\;\oplus\;
\underset{\mathbf 2 \,\otimes\,\tikz \draw (0,0) rectangle (4pt,4pt);}{
  \,\tikz[baseline=(main.base)]{
    \node[draw,rectangle,inner sep=2pt] (main) {%
    \begin{tikzpicture}[baseline=(basepoint)]
    \path (0,-13pt) coordinate (basepoint) (0pt,-1pt) node[dot] {} (8pt,-12pt) node[dot] {} (-4pt,-17pt) node[dot] {};
    \draw (0pt,6pt) -- (0pt,-2pt);
    \draw[] (0pt,0pt) -- (8pt,-12pt);
    \draw[] (8pt,-12pt) -- (-4pt,-17pt);
    \draw[] (0pt,0pt) .. controls +(-6pt,-6pt) and +(-6pt,6pt) .. (-4pt,-17pt);
    \draw (8pt,-12pt) -- (12pt,-24pt);
    \draw (-4pt,-17pt) -- (-4pt,-24pt);
  \end{tikzpicture}\,%
  };
    \draw (main.north) -- ++(0,8pt);
    \draw (main.south) ++(-6pt,0) -- ++(0,-8pt);
    \draw (main.south) ++(6pt,0) -- ++(0,-8pt);
  }\,
  }
$$

Finally, we will need later to know the piece of $(\Frob_{1})^{\shriek}[-1]$ with $\#\mult + \#\comult = 4$ and genus $\beta = 2$.  Since $\#\mult = \beta + m-1 \geq \beta$ and $\#\comult = \beta + n-1 \geq \beta$, the only way to have a composition of $\mult$ and $\comult$ with $\bigl(\#\mult + \#\comult,\beta\bigr) = (4,2)$ is if $\bigl(\#\mult,\#\comult\bigr) = (2,2)$ and $(m,n) = (1,1)$.  Recalling that $
  \,\tikz[baseline=(basepoint)]{
    \path (0,0pt) coordinate (basepoint) (0,1pt) node[dot] {} (0,7pt) node[dot] {};
    \draw (0,-4pt) -- (0,0pt); \draw (0,8pt) -- (0,12pt);
    \draw (0,1pt) .. controls +(4pt,1pt) and +(4pt,-1pt) .. (0,7pt);
    \draw (0,1pt) .. controls +(-4pt,1pt) and +(-4pt,-1pt) .. (0,7pt);
  }\,=0$, the only graph with  $(m,n,\beta) = (1,1,2)$ is 
  \begin{tikzpicture}[baseline=(basepoint)]
    \path (0,12pt) coordinate (basepoint) (0pt,5pt) node[dot] {} (8pt,12pt) node[dot] {} (-4pt,18pt) node[dot] {} (4pt,25pt) node[dot]{};
    \draw (0pt,-2pt) -- (0pt,5pt);
    \draw[] (0pt,5pt) -- (8pt,12pt);
    \draw[] (8pt,12pt) -- (-4pt,18pt);
    \draw[] (0pt,5pt) .. controls +(-6pt,6pt) and +(-6pt,-6pt) .. (-4pt,18pt);
    \draw[] (-4pt,18pt) -- (4pt,25pt);
    \draw[] (8pt,12pt) .. controls +(6pt,6pt) and +(6pt,-6pt) .. (4pt,25pt);
    \draw (4pt,25pt) -- (4pt,32pt);
  \end{tikzpicture}%
  .
    In particular, the (anti)symmetry of the $\mult$ and $\comult$ implies that this graph is equal up to sign to any of its permutations. 
    (Given these signs, it is worth checking that $\begin{tikzpicture}[baseline=(basepoint)]
    \path (0,12pt) coordinate (basepoint) (0pt,5pt) node[dot] {} (8pt,12pt) node[dot] {} (-4pt,18pt) node[dot] {} (4pt,25pt) node[dot]{};
    \draw (0pt,-2pt) -- (0pt,5pt);
    \draw[] (0pt,5pt) -- (8pt,12pt);
    \draw[] (8pt,12pt) -- (-4pt,18pt);
    \draw[] (0pt,5pt) .. controls +(-6pt,6pt) and +(-6pt,-6pt) .. (-4pt,18pt);
    \draw[] (-4pt,18pt) -- (4pt,25pt);
    \draw[] (8pt,12pt) .. controls +(6pt,6pt) and +(6pt,-6pt) .. (4pt,25pt);
    \draw (4pt,25pt) -- (4pt,32pt);
  \end{tikzpicture} \neq 0$ in $\Ff(T[1])$.  The $(\SS_{1}^{\op}\times\SS_{3})$-module $\Ff^{(2)}(\comult)$ splits as $\tikz \draw (0,0) rectangle (4pt,4pt);\otimes\bigl(\tikz \draw (0,0) rectangle (4pt,4pt) (4pt,0pt) rectangle (8pt,4pt) (0pt,-4pt)  rectangle  (4pt,0pt) ; \oplus \tikz \draw (0,0) rectangle (4pt,4pt) (4pt,0pt) rectangle (8pt,4pt) (8pt,0pt)  rectangle  (12pt,4pt) ;\bigr)$, whereas $\Ff^{(2)}(\mult) \cong \bigl(\tikz \draw (0,0) rectangle (4pt,4pt) (4pt,0pt) rectangle (8pt,4pt) (0pt,-4pt)  rectangle  (4pt,0pt) ; \oplus \tikz \draw (0,0) rectangle (4pt,4pt) (0pt,4pt) rectangle (4pt,8pt) (0pt,-4pt)  rectangle  (4pt,0pt) ;\bigr)\otimes \tikz \draw (0,0) rectangle (4pt,4pt);$.  The $(\SS_{1}^{\op}\times \SS_{1})$-module $\Ff^{(2)}(\mult) \otimes_{\SS_{3}} \Ff^{2}(\comult)$ is thus $1$-dimensional, and $\begin{tikzpicture}[baseline=(basepoint)]
    \path (0,12pt) coordinate (basepoint) (0pt,5pt) node[dot] {} (8pt,12pt) node[dot] {} (-4pt,18pt) node[dot] {} (4pt,25pt) node[dot]{};
    \draw (0pt,-2pt) -- (0pt,5pt);
    \draw[] (0pt,5pt) -- (8pt,12pt);
    \draw[] (8pt,12pt) -- (-4pt,18pt);
    \draw[] (0pt,5pt) .. controls +(-6pt,6pt) and +(-6pt,-6pt) .. (-4pt,18pt);
    \draw[] (-4pt,18pt) -- (4pt,25pt);
    \draw[] (8pt,12pt) .. controls +(6pt,6pt) and +(6pt,-6pt) .. (4pt,25pt);
    \draw (4pt,25pt) -- (4pt,32pt);
  \end{tikzpicture}$ is its generator.)
    
   Since \begin{tikzpicture}[baseline=(basepoint)]
    \path (0,12pt) coordinate (basepoint) (0pt,5pt) node[dot] {} (8pt,12pt) node[dot] {} (-4pt,18pt) node[dot] {} (4pt,25pt) node[dot]{};
    \draw (0pt,-2pt) -- (0pt,5pt);
    \draw[] (0pt,5pt) -- (8pt,12pt);
    \draw[] (8pt,12pt) -- (-4pt,18pt);
    \draw[] (0pt,5pt) .. controls +(-6pt,6pt) and +(-6pt,-6pt) .. (-4pt,18pt);
    \draw[] (-4pt,18pt) -- (4pt,25pt);
    \draw[] (8pt,12pt) .. controls +(6pt,6pt) and +(6pt,-6pt) .. (4pt,25pt);
    \draw (4pt,25pt) -- (4pt,32pt);
  \end{tikzpicture} spans a homogeneous piece of $\Ff(T[1])[-1]$, it either is or is not in $(\Frob_{1})^{\shriek}[-1]$, depending on whether its derivative in $\B(\Ff(T[1]))$ is in $\sh\Frob_{1} = \B((\Frob_{1})^{\shriek})$.  In fact:
  \begin{equation} \label{eqn2}
  \partial\left(
  \,\tikz[baseline=(main.base)]{
    \node[draw,rectangle,inner sep=2pt] (main) {
      \begin{tikzpicture}[baseline=(basepoint)]
    \path (0,12pt) coordinate (basepoint) (0pt,3pt) node[dot] {} (8pt,11pt) node[dot] {} (-4pt,19pt) node[dot] {} (4pt,27pt) node[dot]{};
    \draw (0pt,-5pt) -- (0pt,3pt);
    \draw[] (0pt,3pt) -- (8pt,11pt);
    \draw[] (8pt,11pt) -- (-4pt,19pt);
    \draw[] (0pt,3pt) .. controls +(-6pt,6pt) and +(-6pt,-6pt) .. (-4pt,19pt);
    \draw[] (-4pt,19pt) -- (4pt,27pt);
    \draw[] (8pt,11pt) .. controls +(6pt,6pt) and +(6pt,-6pt) .. (4pt,27pt);
    \draw (4pt,27pt) -- (4pt,35pt);
  \end{tikzpicture}
    };
    \draw (main.south) -- ++(0,-8pt);
    \draw (main.north) -- ++(0,8pt);
  }\,  
  \right) = 
  \,\tikz[baseline=(outerbasepoint)]{
    \node[draw,rectangle,inner sep=2pt] (main) {\begin{tikzpicture}[baseline=(basepoint)]
    \path (0,5pt) coordinate (basepoint) (0pt,1pt) node[dot] {} (8pt,12pt) node[dot] {} (-4pt,17pt) node[dot] {};
    \draw (0pt,-6pt) -- (0pt,2pt);
    \draw[] (0pt,0pt) -- (8pt,12pt);
    \draw[] (8pt,12pt) -- (-4pt,17pt);
    \draw[] (0pt,0pt) .. controls +(-6pt,6pt) and +(-6pt,-6pt) .. (-4pt,17pt);
    \draw (8pt,12pt) -- (12pt,24pt);
    \draw (-4pt,17pt) -- (-4pt,24pt);
  \end{tikzpicture}\,};
    \draw (main.south) -- ++(0,-8pt) coordinate(bottom);
    \draw (main.north) ++(-6pt,0) -- ++(0,10pt);
    \draw (main.north) ++(6pt,0) -- ++(0,10pt);
    \draw (main.north) ++(0,18pt) node[draw,inner sep=2pt,fill=white] (upperbox) {\mult};
    \draw (upperbox.north) -- ++(0,10pt) coordinate (top);
    \path (bottom) -- coordinate (outerbasepoint) (top);
  }\,
  \;+\;
  \,\tikz[baseline=(outerbasepoint)]{
    \node[draw,rectangle,inner sep=2pt] (main) {%
    \begin{tikzpicture}[baseline=(basepoint)]
    \path (0,-13pt) coordinate (basepoint) (0pt,-1pt) node[dot] {} (8pt,-12pt) node[dot] {} (-4pt,-17pt) node[dot] {};
    \draw (0pt,6pt) -- (0pt,-2pt);
    \draw[] (0pt,0pt) -- (8pt,-12pt);
    \draw[] (8pt,-12pt) -- (-4pt,-17pt);
    \draw[] (0pt,0pt) .. controls +(-6pt,-6pt) and +(-6pt,6pt) .. (-4pt,-17pt);
    \draw (8pt,-12pt) -- (12pt,-24pt);
    \draw (-4pt,-17pt) -- (-4pt,-24pt);
  \end{tikzpicture}\,%
  };
    \draw (main.north) -- ++(0,8pt) coordinate (top);
    \draw (main.south) ++(-6pt,0) -- ++(0,-10pt);
    \draw (main.south) ++(6pt,0) -- ++(0,-10pt);
    \draw (main.south) ++(0,-18pt) node[draw,inner sep=2pt,fill=white] (lowerbox) {\comult};
    \draw (lowerbox.south) -- ++(0,-10pt) coordinate (bottom);
    \path (bottom) -- coordinate (outerbasepoint) (top);
  }\,
  \;+\; \frac13\;\,
  \begin{tikzpicture}[baseline=(outerbasepoint)]
    \path coordinate (outerbasepoint) ++(0,0)
      +(0,-25pt) node[draw,inner sep = 2pt] (lowerbox) {
        $
        \,\tikz[baseline=(basepoint)]{ 
          \path (0,3pt) coordinate (basepoint) (0,1pt) node[dot] {} (-8pt,11pt) node[dot] {};
          \draw[](0,-9pt) -- (0,1pt);
          \draw[](0,1pt) -- (-8pt,11pt);
          \draw[](-8pt,11pt) -- (-12pt,21pt);
          \draw[](-8pt,11pt) -- (0pt,21pt);
          \draw[](0,1pt) -- (12pt,21pt);
        }\,
        -
        \,\tikz[baseline=(basepoint)]{ 
          \path (0,3pt) coordinate (basepoint) (0,1pt) node[dot] {} (8pt,11pt) node[dot] {};
          \draw[](0,-9pt) -- (0,1pt);
          \draw[](0,1pt) -- (8pt,11pt);
          \draw[](0,1pt) -- (-12pt,21pt);
          \draw[](8pt,11pt) -- (0pt,21pt);
          \draw[](8pt,11pt) -- (12pt,21pt);
        }\,
        $}
      +(0,25pt) node[draw,inner sep = 2pt] (upperbox) {
        $
        \,\tikz[baseline=(basepoint)]{ 
          \path (0,3pt) coordinate (basepoint) (8pt,1pt) node[dot] {} (0,11pt) node[dot] {};
          \draw[](-12pt,-9pt) -- (0,11pt);
          \draw[](0pt,-9pt) -- (8pt,1pt);
          \draw[](12pt,-9pt) -- (8pt,1pt);
          \draw[](8pt,1pt) -- (0,11pt);
          \draw[](0,11pt) -- (0,21pt);
        }\,
        +
        \,\tikz[baseline=(basepoint)]{ 
          \path (0,3pt) coordinate (basepoint) (8pt,1pt) node[dot] {} (0,11pt) node[dot] {};
          \draw[](0pt,-9pt) -- (0,11pt);
          \draw[](12pt,-9pt) -- (8pt,1pt);
          \draw[](-12pt,-9pt) -- (8pt,1pt);
          \draw[](8pt,1pt) -- (0,11pt);
          \draw[](0,11pt) -- (0,21pt);
        }\,        
        $
      };
    \draw (lowerbox.south) -- ++(0,-10pt);
    \draw (upperbox.north) -- ++(0,10pt);
    \draw (lowerbox.north) -- (upperbox.south);
    \path (lowerbox.north) +(12pt,0) coordinate (la) +(-12pt,0) coordinate (lb) (upperbox.south)  +(12pt,0) coordinate (ua) +(-12pt,0) coordinate (ub);
    \draw (la) -- (ua); \draw (lb) -- (ub);
  \end{tikzpicture}
  \;\in\; \sh\Frob_{1}.
  \end{equation}
Lemma~\ref{lemma.cobarinductive} implies therefore that in weight $\#\mult+ \#\comult = 4$, $(\Frob_{1})^{\shriek}[-1]$ looks like:
$$ (\beta = 0 \text{ part}) \;\oplus\; (\beta = 1 \text{ part}) \;\oplus \;
\underset{\tikz \draw (0,0) rectangle (4pt,4pt);\,\otimes\,\tikz \draw (0,0) rectangle (4pt,4pt);}{
  \,\tikz[baseline=(main.base)]{
    \node[draw,rectangle,inner sep=2pt] (main) {
      \begin{tikzpicture}[baseline=(basepoint)]
    \path (0,12pt) coordinate (basepoint) (0pt,5pt) node[dot] {} (8pt,12pt) node[dot] {} (-4pt,18pt) node[dot] {} (4pt,25pt) node[dot]{};
    \draw (0pt,-2pt) -- (0pt,5pt);
    \draw[] (0pt,5pt) -- (8pt,12pt);
    \draw[] (8pt,12pt) -- (-4pt,18pt);
    \draw[] (0pt,5pt) .. controls +(-6pt,6pt) and +(-6pt,-6pt) .. (-4pt,18pt);
    \draw[] (-4pt,18pt) -- (4pt,25pt);
    \draw[] (8pt,12pt) .. controls +(6pt,6pt) and +(6pt,-6pt) .. (4pt,25pt);
    \draw (4pt,25pt) -- (4pt,32pt);
  \end{tikzpicture}
    };
    \draw (main.south) -- ++(0,-8pt);
    \draw (main.north) -- ++(0,8pt);
  }\,  
}$$

\section{The (non)existence of homomorphisms \texorpdfstring{$\sh\Frob_{1} \to \qloc^{\inv}$}{shFrob_1 to QLoc^inv}} \label{section.mainthms}

Let $\sh^{\di}\Frob_{1}$ denote the cofibrant replacement \define{as a dioperad} of $\Frob_{1}$ coming from Corollary~\ref{cor.shFrob}, and let $\sh^{\pr}\Frob_{1}$ denote the cofibrant replacement \define{as a properad}.  There is a canonical inclusion $\sh^{\di}\Frob_{1} \mono \sh^{\pr}\Frob_{1}$, given by including the generators of $\sh^{\di}\Frob_{1}$ as the genus $\beta = 0$ generators of $\sh^{\pr}\Frob_{1}$.  In this final section, we will prove that the space of quasilocal translation-invariant actions of $\sh^{\di}\Frob_{1}$ on $\C_{\bullet}(\RR)$ that induce the standard multiplication on $\H^{1-\bullet}(\RR)$ and the standard comultiplication on $\H_{\bullet}(\RR)$ is contractible, but that such a dioperadic action does not lift to a properad homomorphism $\sh^{\pr}\Frob_{1}\to \qloc$.

\begin{theorem} \label{thm.dioperadic}
  Recall from Lemma~\ref{lemma.cochains} that we have two embeddings $\qloc \mono \End(\C_{\bullet}(\RR))$ and $\qloc \mono \End(\C^{1-\bullet}(\RR))$, and so any $\partial$-closed element of $\qloc$ defines actions on both homology and cohomology.  
  Consider the space of maps $\eta: \sh^{\di}\Frob_{1} \to \qloc^{\inv}$ for which $\H_{\bullet}\bigl(\eta\bigl( \comult\bigr)\bigr) : \H_{\bullet}(\RR) \to \H_{\bullet}(\RR)^{\otimes 2}$ is the standard comultiplication and $\H^{1-\bullet}\bigl( \eta\bigl( \mult\bigr)\bigr) : \H^{1-\bullet}(\RR)^{\otimes 2} \to \H^{1-\bullet}(\RR)$ is the standard multiplication (by working with shifted cohomology $\H^{1-\bullet}$, the degree-$0$ multiplication on $\H^{\bullet}$ shifts to having homological degree $-1 = \deg\bigl(\mult\bigr)$).  This space is contractible, and remains contractible if $\qloc^{\inv}$ is replaced by $\qloc$.
\end{theorem}

By definition, if $P$ and $Q$ are (di/pr)operads, the \define{space of maps $P \to Q$} is the simplicial set whose $k$-simplices are homomorphisms $P \to Q \otimes \QQ[\Delta^{k}]$, where $\QQ[\Delta^{k}] = \QQ[t_{0},\dots,t_{k},\partial t_{0},\dots,\partial t_{k}] / \bigl\<\sum t_{i} = 1,\, \sum \partial t_{i} = 0\bigr\>$ is the commutative dg algebra of polynomial de Rham forms on the $k$-dimensional simplex.  If $P$ is cofibrant, then this simplicial set satisfies the Kan horn-filling condition.  By convention, a space is \define{contractible} if it has the homotopy type of $\{*\}$, which in particular includes that it is nonempty.

To prove Theorem~\ref{thm.dioperadic}, we will use the following reasonably well known \define{basic facts of obstruction theory}.  Let $P$ be a quasifree (di/pr)operad with a well-ordered set of generators $f$ (each of which is  homogeneous  of homological degree $\deg(f)$), and such that for each generator $f$ of $P$, $\partial(f)$ is a composition of generators of $P$ that are strictly earlier than $f$ in the well-ordering.  For each generator $f$, let $P_{< f}$ denote the sub(di/pr)operad of $P$ generated by the generators that are strictly earlier than $f$, and $P_{\leq f}$ the sub(di/pr)operad generated $P_{<f}$ and $f$; the condition then guarantees that $P_{<f}$ and $P_{\leq f}$ are dg (di/pr)operads and that $\partial(f) \in P_{<f}$.  The basic facts that we will use provide ways to study, for any (di/pr)operad $Q$, the space of homomorphisms $\eta : P \to Q$.

\begin{enumerate}
  \item Homomorphisms $\eta: P \to Q$ may be built and studied inductively.  Suppose we have defined a homomorphism $\eta_{<f} : P_{<f} \to Q$, which we want to extend to $P_{\leq f}$.  Then $\eta(\partial f)\in Q$ is closed of homological degree $\deg(f)-1$.  The \define{obstruction to defining $f$} is the class of $\eta(\partial f)$ in $\H_{\bullet}(Q)$.  The first basic fact of obstruction theory is that $\eta(f)$ can be defined, and therefore the induction can be continued, if and only if its obstruction vanishes.  In particular, maps $P \to Q$ are easy to construct whenever the homology groups of $Q$ vanish in degrees $\deg(f)-1$ for generators $f$ of $P$.
  \item There are, of course, many choices for $\eta(f)$ --- as many as there are closed elements of $Q$ (with the appropriate number of inputs and outputs) of degree $\deg (f)$.  Different choices might lead to later steps of the induction succeeding or failing.  The second basic fact says that whether a later step succeeds is not affected by changing $\eta(f)$ by something exact, so that the ``true'' space of choices for $\eta(f)$ (provided the obstruction is exact) is a torsor for $\H_{\deg(f)}(Q)$.
  
  To prove this, note first that two different extensions of $\eta_{<f} : P_{<f}\to Q$ whose values on $f$ differ by something exact can be connected by a ``linear'' path $P_{\geq f} \to Q\otimes \QQ[\Delta^{1}] = Q\otimes\QQ[t]$ (and conversely one may check that the endpoints of any path $P_{\leq f} \to Q\otimes \QQ[\Delta^{1}]$ which is constant on $P_{<f}$ assigns to $f$ values that differ by something exact).  Suppose then by induction that we have a curve $\eta_{<f'} : P_{< f'} \to Q\otimes \QQ[\Delta^{1}]$, and that we choose an extension of its left endpoint to $\eta_{\leq f'}|_{t=0} : P_{\leq f'} \to Q$.  It suffices to define $\eta_{\leq f'} : P_{\leq f'} \to Q\otimes \QQ[\Delta^{1}]$ extending $\eta_{<f'}$ with the given boundary condition.  
  To find such an extension $\eta_{\leq f'}$, one may  use~\cite[Proposition~37]{MR2572248}, which implies that the inclusion $P_{<f'}\mono P_{\leq f'}$ is a cofibration; the projection $Q\otimes \QQ[\Delta^{1}] \to Q$ that sets $t = 0$ is a surjective quasiisomorphism, i.e.\ an acyclic fibration; the extension then exists by the left lifting property.
  
  \item Finally, a similar analysis shows that, more generally, $\pi_{j}\{$space of choices for $\eta_{\leq f}$ extending $\eta_{<f}\} = \H_{\deg(f)+j}(Q)$, provided $\pi_{i} = 0$ for $i<j$.  Here the correct conventions are: $\pi_{-1}(X) = 0 = \{*\} =$ ``true'' for $X$ nonempty, and $\pi_{-1}(\emptyset) = \emptyset = $ ``false''; and that $\pi_{j}(X) = 0$ for all $X$ and all $j<-1$.  Thus in particular asking that $\pi_{-1}\{\text{choices}\}=0$ is the same as asking that the obstruction is exact, so that a choice exists.

 In particular, we may immediately conclude that the space of choices for $\eta(f)$ is contractible whenever $\H_{\deg(f) + j}(Q) = 0$ for all $j \geq -1$.
\end{enumerate}

\begin{proof}[Proof of Theorem~\ref{thm.dioperadic}]
  Consider first the generators $\mult$ and $\comult$ of $\sh^{\di}\Frob_{1}$.  In order for them to induce the standard (co)multiplications on (co)homology, they must each be assigned to  Thom forms around $\diag(\RR) \subseteq \RR^{3}$.  The space of Thom forms is contractible, and remains so upon imposing translation invariance.
  
  For the remaining generators, the homotopy type of the space of choices is contractible, as we now show.  We computed in Proposition~\ref{homologyofqloc} that $\qloc^{\inv}(m,n)$ has homology only in degrees $-m$ and $-m+1$, and that $\qloc(m,n)$ has homology only in degree $-m+1$.  Let $f$ be a generator of $\sh^{\di}\Frob_{1}$; we may suppose without loss of generality that it is homogeneous for the bigrading $\bigl(\#\mult,\#\comult\bigr)$.  Then $f$ is in homological degree $\deg(f) = \#\comult - 1$ by Corollary~\ref{cor.shFrob}.  Thus the obstruction might be nonzero only when $\# \comult - 2 = -m$ or $-m+1$.  Recalling that $m = \#\mult + 1$, we see that
   {the obstruction automatically vanishes unless }$\#\mult + \#\comult = 1 \text{ or }2. $
  
  We have addressed already the generators for which $\#\mult + \#\comult = 1$.  For the generators with $\#\mult + \#\comult = 2$, Lemma~\ref{lemma.hqlocfrob} implies that the obstruction is the difference between two translation-invariant Thom forms around $\diag(\RR) \subseteq \RR^{4}$, and hence vanishes in homology.
\end{proof}

Theorem~\ref{thm.dioperadic} gives an affirmative answer to the question of geometrically lifting the \emph{dioperadic} action of $\Frob_{1}$ on $\H_{\bullet}(\RR)$ to the chain level.  The main result of this paper is that the similar question for \emph{properads} has a negative answer.

\begin{theorem} \label{thm.properadic}
  There does not exist a properad homomorphism $\eta: \sh^{\pr}\Frob_{1} \to \qloc$ such that $\H_{\bullet}\bigl(\eta\bigl( \comult\bigr)\bigr)$ and $\H^{1-\bullet}\bigl(\eta(\bigl(\mult\bigr))$ are the standard comultiplication and multiplication.
\end{theorem}

In light of Lemma~\ref{lemma.hqlocfrob}, Theorems~\ref{thm.dioperadic} and~\ref{thm.properadic} may be rephrased as saying that $\qloc$ is \define{formal}, i.e.\ quasiisomorphic to its homology, as a dioperad but not as a properad.

\begin{proof}
  Let $f$ be a generator of $\sh^{\pr}\Frob_{1}$ with $m$ inputs and $n$ outputs which is homogeneous for the bigrading $\bigl(\#\mult,\#\comult\bigr)$; 
  as in Corollary~\ref{cor.shFrob}, we
   let $\beta = \#\mult -m + 1 = \#\comult - n + 1$ denote its genus.  Restricting $\eta$ just to the genus $\beta = 0$ generators determines a map $\sh^{\di}\Frob_{1}\to \qloc$, and by Theorem~\ref{thm.dioperadic} the space of such maps is contractible.  It then follows from the basic facts of obstruction theory that the homotopy type of the space of extensions to a map $\sh^{\pr}\Frob_{1}\to \qloc$ is independent of the choices made for the map $\sh^{\di}\Frob_{1}\to \qloc$.
  
  Simple combinatorics implies moreover that the obstruction automatically vanishes unless $m+n + \beta = 4$, and that the space of choices is contractible if nonempty.  Thus the only  generators of $\sh^{\pr}\Frob_{1}$ beyond those of $\sh^{\di}\Frob_{1}$ for which the obstruction might not vanish are permutations of $  \,\tikz[baseline=(main.base)]{
    \node[draw,rectangle,inner sep=2pt] (main) {\begin{tikzpicture}[baseline=(basepoint)]
    \path (0,5pt) coordinate (basepoint) (0pt,1pt) node[dot] {} (8pt,12pt) node[dot] {} (-4pt,17pt) node[dot] {};
    \draw (0pt,-6pt) -- (0pt,2pt);
    \draw[] (0pt,0pt) -- (8pt,12pt);
    \draw[] (8pt,12pt) -- (-4pt,17pt);
    \draw[] (0pt,0pt) .. controls +(-6pt,6pt) and +(-6pt,-6pt) .. (-4pt,17pt);
    \draw (8pt,12pt) -- (12pt,24pt);
    \draw (-4pt,17pt) -- (-4pt,24pt);
  \end{tikzpicture}\,};
    \draw (main.south) -- ++(0,-8pt);
    \draw (main.north) ++(-6pt,0) -- ++(0,8pt);
    \draw (main.north) ++(6pt,0) -- ++(0,8pt);
  }\,$, $  \,\tikz[baseline=(main.base)]{
    \node[draw,rectangle,inner sep=2pt] (main) {%
    \begin{tikzpicture}[baseline=(basepoint)]
    \path (0,-13pt) coordinate (basepoint) (0pt,-1pt) node[dot] {} (8pt,-12pt) node[dot] {} (-4pt,-17pt) node[dot] {};
    \draw (0pt,6pt) -- (0pt,-2pt);
    \draw[] (0pt,0pt) -- (8pt,-12pt);
    \draw[] (8pt,-12pt) -- (-4pt,-17pt);
    \draw[] (0pt,0pt) .. controls +(-6pt,-6pt) and +(-6pt,6pt) .. (-4pt,-17pt);
    \draw (8pt,-12pt) -- (12pt,-24pt);
    \draw (-4pt,-17pt) -- (-4pt,-24pt);
  \end{tikzpicture}\,%
  };
    \draw (main.north) -- ++(0,8pt);
    \draw (main.south) ++(-6pt,0) -- ++(0,-8pt);
    \draw (main.south) ++(6pt,0) -- ++(0,-8pt);
  }\,$, and $  \,\tikz[baseline=(main.base)]{
    \node[draw,rectangle,inner sep=2pt] (main) {
      \begin{tikzpicture}[baseline=(basepoint)]
    \path (0,12pt) coordinate (basepoint) (0pt,5pt) node[dot] {} (8pt,12pt) node[dot] {} (-4pt,18pt) node[dot] {} (4pt,25pt) node[dot]{};
    \draw (0pt,-2pt) -- (0pt,5pt);
    \draw[] (0pt,5pt) -- (8pt,12pt);
    \draw[] (8pt,12pt) -- (-4pt,18pt);
    \draw[] (0pt,5pt) .. controls +(-6pt,6pt) and +(-6pt,-6pt) .. (-4pt,18pt);
    \draw[] (-4pt,18pt) -- (4pt,25pt);
    \draw[] (8pt,12pt) .. controls +(6pt,6pt) and +(6pt,-6pt) .. (4pt,25pt);
    \draw (4pt,25pt) -- (4pt,32pt);
  \end{tikzpicture}
    };
    \draw (main.south) -- ++(0,-8pt);
    \draw (main.north) -- ++(0,8pt);
  }\,  
$.  We will calculate their obstructions.

We will focus on the cellular model of chains $\C_{\bullet}(\RR)$.  The result for the de Rham models follows by replacing the choices below by smooth approximations thereof.

  For the cellular model, there is a $0$-cell $c_{z}$ for each $z\in \ZZ$, and a $1$-cell $c_{z+\frac12}$ for each $z\in \ZZ$, and the boundary operator is $\partial c_{z+\frac12} = c_{z+1} - c_{z}$.  We might as well choose the following Thom forms:
\begin{gather*}
  \,\tikz[baseline=(main.base)]{
    \node[draw,rectangle,inner sep=2pt] (main) {\comult};
    \draw (main.south) -- ++(0,-8pt);
    \draw (main.north) ++(-6pt,0) -- ++(0,8pt);
    \draw (main.north) ++(6pt,0) -- ++(0,8pt);
  }\,
  :  \quad
  c_{z} \mapsto c_{z}\otimes c_{z}, \quad c_{z+\frac12} \mapsto \frac12 \left( (c_{z}+ c_{z+1}) \otimes c_{z+\frac12} + c_{z+\frac12}\otimes (c_{z} + c_{z+1}) \right), 
  \\
  \,\tikz[baseline=(main.base)]{
    \node[draw,rectangle,inner sep=2pt] (main) {\mult};
    \draw (main.north) -- ++(0,8pt);
    \draw (main.south) ++(-6pt,0) -- ++(0,-8pt);
    \draw (main.south) ++(6pt,0) -- ++(0,-8pt);
  }\,
  :\quad
  c_{z+\frac12}\otimes c_{z+\frac12} \mapsto c_{z+\frac12}, \quad c_{z}\otimes c_{z\pm \frac12}\mapsto -\frac12 c_{z}, \quad c_{z\pm\frac12}\otimes c_{z} \mapsto \frac12 c_{z}.
\end{gather*}
In both lines $z\in \ZZ$, and in the second line all non-listed pairs $c_{x}\otimes c_{y}$ get mapped to $0$.  

With these choices, $\partial   
  \,\tikz[baseline=(main.base)]{
    \node[draw,rectangle,inner sep=2pt] (main) {$
        \,\tikz[baseline=(basepoint)]{ 
    \path (0,3pt) coordinate (basepoint) (0,1pt) node[dot] {} (-8pt,11pt) node[dot] {};
    \draw[](0,-9pt) -- (0,1pt);
    \draw[](0,1pt) -- (-8pt,11pt);
    \draw[](-8pt,11pt) -- (-12pt,21pt);
    \draw[](-8pt,11pt) -- (0pt,21pt);
    \draw[](0,1pt) -- (12pt,21pt);
  }\,
  -
  \,\tikz[baseline=(basepoint)]{ 
    \path (0,3pt) coordinate (basepoint) (0,1pt) node[dot] {} (8pt,11pt) node[dot] {};
    \draw[](0,-9pt) -- (0,1pt);
    \draw[](0,1pt) -- (8pt,11pt);
    \draw[](0,1pt) -- (-12pt,21pt);
    \draw[](8pt,11pt) -- (0pt,21pt);
    \draw[](8pt,11pt) -- (12pt,21pt);
  }\,
    $};
    \draw (main.south) -- ++(0,-8pt);
    \draw (main.north) ++(-12pt,0) -- ++(0,8pt);
    \draw (main.north) ++(0pt,0) -- ++(0,8pt);
    \draw (main.north) ++(12pt,0) -- ++(0,8pt);
  }\,$
  sends $c_{z}\mapsto 0$ when $z\in \ZZ$, and sends: \vspace*{-4pt}
  $$ c_{z+\frac12}\mapsto \frac14 \left( (c_{z+1}-c_{z}) \otimes (c_{z+1}-c_{z}) \otimes c_{z+\frac12} - c_{z+\frac12} \otimes (c_{z+1}-c_{z}) \otimes (c_{z+1}-c_{z}) \right).\vspace*{-4pt}$$
  Any assignment for the generator \,\tikz[baseline=(main.base)]{
    \node[draw,rectangle,inner sep=2pt] (main) {$
        \,\tikz[baseline=(basepoint)]{ 
    \path (0,3pt) coordinate (basepoint) (0,1pt) node[dot] {} (-8pt,11pt) node[dot] {};
    \draw[](0,-9pt) -- (0,1pt);
    \draw[](0,1pt) -- (-8pt,11pt);
    \draw[](-8pt,11pt) -- (-12pt,21pt);
    \draw[](-8pt,11pt) -- (0pt,21pt);
    \draw[](0,1pt) -- (12pt,21pt);
  }\,
  -
  \,\tikz[baseline=(basepoint)]{ 
    \path (0,3pt) coordinate (basepoint) (0,1pt) node[dot] {} (8pt,11pt) node[dot] {};
    \draw[](0,-9pt) -- (0,1pt);
    \draw[](0,1pt) -- (8pt,11pt);
    \draw[](0,1pt) -- (-12pt,21pt);
    \draw[](8pt,11pt) -- (0pt,21pt);
    \draw[](8pt,11pt) -- (12pt,21pt);
  }\,
    $};
    \draw (main.south) -- ++(0,-8pt);
    \draw (main.north) ++(-12pt,0) -- ++(0,8pt);
    \draw (main.north) ++(0pt,0) -- ++(0,8pt);
    \draw (main.north) ++(12pt,0) -- ++(0,8pt);
  }\, itself must have this as its derivative, and must transform in the two-dimensional irrep \tikz \draw (0,0) rectangle (4pt,4pt) (4pt,0pt) rectangle (8pt,4pt) (0pt,-4pt)  rectangle  (4pt,0pt) ; of $\SS_{3}$.  One choice that works is to send $c_{z}\mapsto 0$ for $z\in \ZZ$ and:
  $$ 
  c_{z+\frac12} \mapsto \frac16 c_{z+\frac12} \otimes (c_{z+1}-c_{z}) \otimes c_{z+\frac12} + \frac1{12} \left(
   (c_{z+1}-c_{z}) \otimes c_{z+\frac12} \otimes c_{z+\frac12} + c_{z+\frac12} \otimes c_{z+\frac12} \otimes  (c_{z+1}-c_{z})
   \right)
  $$
  By our obstruction-theoretic analysis, the particular choice won't affect whether later obstructions vanish.  \vspace*{-8pt}
  
  A similar computation implies that the generator 
  \,\tikz[baseline=(main.base)]{
    \node[draw,rectangle,inner sep=2pt] (main) {$
          \,\tikz[baseline=(basepoint)]{ 
    \path (0,3pt) coordinate (basepoint) (0,1pt) node[dot] {} (0,11pt) node[dot] {};
    \draw[](-8pt,-9pt) -- (0,1pt);
    \draw[](8pt,-9pt) -- (0,1pt);
    \draw[](0,1pt) -- (0,11pt);
    \draw[](0,11pt) -- (-8pt,21pt);
    \draw[](0,11pt) -- (8pt,21pt);
  }\,
  -
  \,\tikz[baseline=(basepoint)]{ 
    \path (0,3pt) coordinate (basepoint) (8pt,1pt) node[dot] {} (-8pt,11pt) node[dot] {};
    \draw[](-8pt,-9pt) -- (-8pt,11pt);
    \draw[](8pt,-9pt) -- (8pt,1pt);
    \draw[](8pt,1pt) -- (-8pt,11pt);
    \draw[](-8pt,11pt) -- (-8pt,21pt);
    \draw[](8pt,1pt) -- (8pt,21pt);
  }\,
    $};
    \draw (main.south) ++(-6pt,0) -- ++(0pt,-8pt);
    \draw (main.south) ++(6pt,0) -- ++(0pt,-8pt);
    \draw (main.north) ++(-6pt,0) -- ++(0,8pt);
    \draw (main.north) ++(6pt,0) -- ++(0,8pt);
  }\,
  may be chosen to send
  $$   
  c_{z} \otimes c_{z\pm \frac12} \mapsto \pm \frac14 c_{z} \otimes c_{z\pm \frac12},\; \forall z\in \ZZ, \quad \text{ all other }c_{x}\otimes c_{y} \mapsto 0. $$
  
  We may now compute the action of $\partial \,\tikz[baseline=(main.base)]{
    \node[draw,rectangle,inner sep=2pt] (main) {\begin{tikzpicture}[baseline=(basepoint)]
    \path (0,5pt) coordinate (basepoint) (0pt,1pt) node[dot] {} (8pt,12pt) node[dot] {} (-4pt,17pt) node[dot] {};
    \draw (0pt,-6pt) -- (0pt,2pt);
    \draw[] (0pt,0pt) -- (8pt,12pt);
    \draw[] (8pt,12pt) -- (-4pt,17pt);
    \draw[] (0pt,0pt) .. controls +(-6pt,6pt) and +(-6pt,-6pt) .. (-4pt,17pt);
    \draw (8pt,12pt) -- (12pt,24pt);
    \draw (-4pt,17pt) -- (-4pt,24pt);
  \end{tikzpicture}\,};
    \draw (main.south) -- ++(0,-8pt);
    \draw (main.north) ++(-6pt,0) -- ++(0,8pt);
    \draw (main.north) ++(6pt,0) -- ++(0,8pt);
  }\,$, using equation~\eqref{eqn1}. \vspace*{-8pt}  With our choices, this derivative vanishes on $c_{z}$ for $z\in \ZZ$, and sends:
  $$ c_{z+\frac12}\mapsto -\frac1{12} (c_{z+1} - c_{z})\otimes c_{z+\frac12} + \frac1{12} c_{z+\frac12} \otimes (c_{z+1} - c_{z}).  $$
  Since the $\SS_{2}$ action is free, we may choose:
  $$ \,\tikz[baseline=(main.base)]{
    \node[draw,rectangle,inner sep=2pt] (main) {\begin{tikzpicture}[baseline=(basepoint)]
    \path (0,5pt) coordinate (basepoint) (0pt,1pt) node[dot] {} (8pt,12pt) node[dot] {} (-4pt,17pt) node[dot] {};
    \draw (0pt,-6pt) -- (0pt,2pt);
    \draw[] (0pt,0pt) -- (8pt,12pt);
    \draw[] (8pt,12pt) -- (-4pt,17pt);
    \draw[] (0pt,0pt) .. controls +(-6pt,6pt) and +(-6pt,-6pt) .. (-4pt,17pt);
    \draw (8pt,12pt) -- (12pt,24pt);
    \draw (-4pt,17pt) -- (-4pt,24pt);
  \end{tikzpicture}\,};
    \draw (main.south) -- ++(0,-8pt);
    \draw (main.north) ++(-6pt,0) -- ++(0,8pt);
    \draw (main.north) ++(6pt,0) -- ++(0,8pt);
  }\, : c_{z}\mapsto 0 \text{ and } c_{z+\frac12} \mapsto -\frac1{12} c_{z+\frac12} \otimes c_{z+\frac12},\quad \forall z\in \ZZ.\vspace*{-16pt}$$
  By making dual choices, we find in the same way a value for $  \,\tikz[baseline=(main.base)]{
    \node[draw,rectangle,inner sep=2pt] (main) {%
    \begin{tikzpicture}[baseline=(basepoint)]
    \path (0,-13pt) coordinate (basepoint) (0pt,-1pt) node[dot] {} (8pt,-12pt) node[dot] {} (-4pt,-17pt) node[dot] {};
    \draw (0pt,6pt) -- (0pt,-2pt);
    \draw[] (0pt,0pt) -- (8pt,-12pt);
    \draw[] (8pt,-12pt) -- (-4pt,-17pt);
    \draw[] (0pt,0pt) .. controls +(-6pt,-6pt) and +(-6pt,6pt) .. (-4pt,-17pt);
    \draw (8pt,-12pt) -- (12pt,-24pt);
    \draw (-4pt,-17pt) -- (-4pt,-24pt);
  \end{tikzpicture}\,%
  };
    \draw (main.north) -- ++(0,8pt);
    \draw (main.south) ++(-6pt,0) -- ++(0,-8pt);
    \draw (main.south) ++(6pt,0) -- ++(0,-8pt);
  }\,$.
  
  The last generator for which the obstruction might not vanish is 
  \,\tikz[baseline=(main.base)]{
    \node[draw,rectangle,inner sep=2pt] (main) {
      \begin{tikzpicture}[baseline=(basepoint)]
    \path (0,12pt) coordinate (basepoint) (0pt,5pt) node[dot] {} (8pt,12pt) node[dot] {} (-4pt,18pt) node[dot] {} (4pt,25pt) node[dot]{};
    \draw (0pt,-2pt) -- (0pt,5pt);
    \draw[] (0pt,5pt) -- (8pt,12pt);
    \draw[] (8pt,12pt) -- (-4pt,18pt);
    \draw[] (0pt,5pt) .. controls +(-6pt,6pt) and +(-6pt,-6pt) .. (-4pt,18pt);
    \draw[] (-4pt,18pt) -- (4pt,25pt);
    \draw[] (8pt,12pt) .. controls +(6pt,6pt) and +(6pt,-6pt) .. (4pt,25pt);
    \draw (4pt,25pt) -- (4pt,32pt);
  \end{tikzpicture}
    };
    \draw (main.south) -- ++(0,-8pt);
    \draw (main.north) -- ++(0,8pt);
  }\,%
  .  Our computations above imply that in equation~\eqref{eqn2}, the first summand on the right hand side acts by $-\frac1{12}$ on $1$-chains and by $0$ on $0$-chains.  The second summand acts by $-\frac1{12}$ on $0$-chains and by $0$ on $1$-chains by a similar computation; indeed, we may make all choices symmetrically with respect to reversing inputs and outputs, and simultaneously reversing $0$-chains and $1$-chains.\vspace*{-20pt}  Finally, \,\tikz[baseline=(main.base)]{
    \node[draw,rectangle,inner sep=2pt] (main) {$
        \,\tikz[baseline=(basepoint)]{ 
    \path (0,3pt) coordinate (basepoint) (0,1pt) node[dot] {} (-8pt,11pt) node[dot] {};
    \draw[](0,-9pt) -- (0,1pt);
    \draw[](0,1pt) -- (-8pt,11pt);
    \draw[](-8pt,11pt) -- (-12pt,21pt);
    \draw[](-8pt,11pt) -- (0pt,21pt);
    \draw[](0,1pt) -- (12pt,21pt);
  }\,
  -
  \,\tikz[baseline=(basepoint)]{ 
    \path (0,3pt) coordinate (basepoint) (0,1pt) node[dot] {} (8pt,11pt) node[dot] {};
    \draw[](0,-9pt) -- (0,1pt);
    \draw[](0,1pt) -- (8pt,11pt);
    \draw[](0,1pt) -- (-12pt,21pt);
    \draw[](8pt,11pt) -- (0pt,21pt);
    \draw[](8pt,11pt) -- (12pt,21pt);
  }\,
    $};
    \draw (main.south) -- ++(0,-8pt);
    \draw (main.north) ++(-12pt,0) -- ++(0,8pt);
    \draw (main.north) ++(0pt,0) -- ++(0,8pt);
    \draw (main.north) ++(12pt,0) -- ++(0,8pt);
  }\, is non-zero only on $1$-chains, but the image of   \,\tikz[baseline=(main.base)]{
    \node[draw,rectangle,inner sep=2pt] (main) {$
      \,\tikz[baseline=(basepoint)]{ 
    \path (0,3pt) coordinate (basepoint) (8pt,1pt) node[dot] {} (0,11pt) node[dot] {};
    \draw[](-12pt,-9pt) -- (0,11pt);
    \draw[](0pt,-9pt) -- (8pt,1pt);
    \draw[](12pt,-9pt) -- (8pt,1pt);
    \draw[](8pt,1pt) -- (0,11pt);
    \draw[](0,11pt) -- (0,21pt);
  }\,
  +
  \,\tikz[baseline=(basepoint)]{ 
    \path (0,3pt) coordinate (basepoint) (-8pt,1pt) node[dot] {} (0,11pt) node[dot] {};
    \draw[](12pt,-9pt) -- (0,11pt);
    \draw[](-12pt,-9pt) -- (-8pt,1pt);
    \draw[](0pt,-9pt) -- (-8pt,1pt);
    \draw[](-8pt,1pt) -- (0,11pt);
    \draw[](0,11pt) -- (0,21pt);
  }\,
    $};
    \draw (main.north) -- ++(0,8pt);
    \draw (main.south) ++(-12pt,0) -- ++(0,-8pt);
    \draw (main.south) ++(0pt,0) -- ++(0,-8pt);
    \draw (main.south) ++(12pt,0) -- ++(0,-8pt);
  }\, consists only of $0$-chains, and so the last summand in equation~\eqref{eqn2} vanishes, as it describes a map of homological degree $0$.
  
  All together, we see that $\partial\left(
  \,\tikz[baseline=(main.base)]{
    \node[draw,rectangle,inner sep=2pt] (main) {
      \begin{tikzpicture}[baseline=(basepoint)]
    \path (0,12pt) coordinate (basepoint) (0pt,5pt) node[dot] {} (8pt,12pt) node[dot] {} (-4pt,18pt) node[dot] {} (4pt,25pt) node[dot]{};
    \draw (0pt,-2pt) -- (0pt,5pt);
    \draw[] (0pt,5pt) -- (8pt,12pt);
    \draw[] (8pt,12pt) -- (-4pt,18pt);
    \draw[] (0pt,5pt) .. controls +(-6pt,6pt) and +(-6pt,-6pt) .. (-4pt,18pt);
    \draw[] (-4pt,18pt) -- (4pt,25pt);
    \draw[] (8pt,12pt) .. controls +(6pt,6pt) and +(6pt,-6pt) .. (4pt,25pt);
    \draw (4pt,25pt) -- (4pt,32pt);
  \end{tikzpicture}
    };
    \draw (main.south) -- ++(0,-8pt);
    \draw (main.north) -- ++(0,8pt);
  }\,
  \right)$
  acts on $\C_{\bullet}(\RR)$ by multiplication by $-\frac1{12}$.  The identity operator is exact only on complexes with trivial homology, and in particular multiplication by $-\frac1{12}$ is not exact in $\qloc$.  It therefore obstructs the existence of a homomorphism $\sh^{\pr}\Frob_{1} \to \qloc$ sending $\mult$ and $\comult$ to Thom forms.
\end{proof}

As we remarked in the introduction, homotopy perturbation theory does construct a homotopy $\Frob_{1}$ structure on $\C_{\bullet}(\RR)$ inducing the structure on $\H_{\bullet}(\RR)$.  It may even be chosen quasilocally: it then satisfies $\H^{1-\bullet}(\mult) = 0$, not the standard multiplication.  Alternately, one can set $\mult$ to the intersection of chains to try to force $\H^{1-\bullet}(\mult)$ to be the standard multiplication; but then necessarily the generator 
$\,\tikz[baseline=(main.base)]{
    \node[draw,rectangle,inner sep=2pt] (main) {\begin{tikzpicture}[baseline=(basepoint)]
    \path (0,5pt) coordinate (basepoint) (0pt,1pt) node[dot] {} (8pt,12pt) node[dot] {} (-4pt,17pt) node[dot] {};
    \draw (0pt,-6pt) -- (0pt,2pt);
    \draw[] (0pt,0pt) -- (8pt,12pt);
    \draw[] (8pt,12pt) -- (-4pt,17pt);
    \draw[] (0pt,0pt) .. controls +(-6pt,6pt) and +(-6pt,-6pt) .. (-4pt,17pt);
    \draw (8pt,12pt) -- (12pt,24pt);
    \draw (-4pt,17pt) -- (-4pt,24pt);
  \end{tikzpicture}\,};
    \draw (main.south) -- ++(0,-8pt);
    \draw (main.north) ++(-6pt,0) -- ++(0,8pt);
    \draw (main.north) ++(6pt,0) -- ++(0,8pt);
  }\,$
  will act non-quasilocally.


\end{document}